\documentclass[8pt]{elsarticle}

\usepackage{lineno,hyperref}

\RequirePackage{times}
\RequirePackage{graphics}
\RequirePackage{afterpage}
\usepackage{color}
\usepackage{mathtools}
\usepackage{amssymb}
\usepackage{amsmath}
\usepackage{latexsym}
\usepackage{hyperref}
\usepackage{bm}
\usepackage{graphics}
\usepackage{graphicx}  
\usepackage{epstopdf} 
\usepackage{times}
\usepackage{multirow}
\usepackage{times}
\usepackage{algorithm}
\usepackage{algpseudocode}
\usepackage{pifont}
\usepackage{multirow}
\usepackage[cal=boondoxo]{mathalfa}
\usepackage{amsthm}
\usepackage{mathdefs}

\usepackage{amsmath,amssymb,graphicx}
\usepackage{amsthm,verbatim}
\usepackage{amsthm}
\usepackage{bbm}
\usepackage{textcomp}

\newcommand{\bs}[1]{\boldsymbol{#1}}

\begin{document}

\begin{frontmatter}

\title{\bf{Generation of Nested Quadrature Rules for Generic Weight Functions via Numerical Optimization: Application to Sparse Grids}}

\cortext[Cauthor]{Corresponding author}
\author[Vahid]{Vahid Keshavarzzadeh\corref{Cauthor}}
\ead{vkeshava@sci.utah.edu}

\author{Robert M. Kirby\fnref{my1,my2}}
\ead{kirby@sci.utah.edu}

\author{Akil Narayan\fnref{my1,my3}}
\ead{akil@sci.utah.edu}

\fntext[my1]{Scientific Computing and Imaging Institute, University of Utah}
\fntext[my2]{School of Computing, University of Utah}
\fntext[my3]{Department of Mathematics, University of Utah}

\begin{abstract}
We present a numerical framework for computing nested quadrature rules for various weight functions. The well-known Kronrod method extends the Gauss-Legendre quadrature by adding new optimal nodes to the existing Gauss nodes for integration of higher order polynomials. Our numerical method generalizes the Kronrod rule for any continuous probability density function on real line with finite moments. We develop a bi-level optimization scheme to solve moment-matching conditions for two levels of main and nested rule and use a penalty method to enforce the constraints on the limits of the nodes and weights. We demonstrate our nested quadrature rule for probability measures on finite/infinite and symmetric/asymmetric supports. We generate Gauss-Kronrod-Patterson rules by slightly modifying our algorithm and present results associated with Chebyshev polynomials which are not reported elsewhere. We finally show the application of our nested rules in construction of sparse grids where we validate the accuracy and efficiency of such nested quadrature-based sparse grids on parameterized boundary and initial value problems in multiple dimensions.
\end{abstract}


\end{frontmatter}

\section{Introduction}

A quadrature formula for integration takes the form
\begin{equation*}
\int_{\Gamma} f(x) \omega(x) dx = \sum_{i=1}^n w_i f(x_i),
\end{equation*}
where the weight $\omega:\Gamma \rightarrow R$ is a positive measurable function with finite moments. The weights $w_i$ and nodes $x_i$
are selected to maximize the order $p$ for which the above integral is exact for all polynomials $f$ with degree up to $p$. The well-known optimal
rule for integration in one variable is Gauss quadrature which integrates polynomials of degree $p=2n-1$ or less with $n$ nodes.

In many scientific computing applications such as uncertainty quantification (UQ), integrals are evaluated via quadrature rules where each
quadrature node typically corresponds to an expensive simulation. When integrals with differing accuracies (i.e., polynomial exactness) are desired for error estimation or extrapolation,
it is therefore desirable to use a nested quadrature rule where each formula is a subset of a node set with higher degree of exactness.
A generic strategy is to start with a $n_1$ rule and enrich it with $n_2-n_1 > 0$ nodes where the resulting $n_2$
nodes with a new set of weights integrates higher order polynomials.

Kronrod~\cite{Kronrod64} extended the well-known Gauss-Legendre formulas by adding $n+1$ points to existing $n$ Gauss points in some cases. The resulting nodes integrate
with accuracy $p = 3n+1$ for $n$ even, and $p = 3n+2$ for $n$ odd. He showed that this is the best possible extension in terms
of the maximum degree of exactness and provided tables for up to $n=40$ points.

In this paper we propose a systematic optimization algorithm that generates nested quadrature formulas for general
continuous univariate distributions. Our algorithm is simple and easily implemented in which we satisfy moment-matching conditions that are governed by three-term recurrence rules for orthogonal polynomials. The complete pseudocode is provided in this paper; we plan to distribute the MATLAB implementation of our algorithm in a public repository in the future.


The organization of this paper is as follows. In Section 2 we briefly discuss mathematical setting of Gauss quadrature, Kronrod's extension for nested quadrature and our nested rule. We describe in detail the computational framework for generating the proposed nested quadrature rule in Section 3, and present numerical results in Section 4. Finally, Section 5 discusses the concluding remarks.

\section{Univariate Quadrature}

\subsection{Notation}\label{sec_notation}

Let $\omega(\bm{x})$ be a given non-negative weight function (or a probability density function) whose support is $\Gamma \subset \R$ where $\Gamma$ need not be compact. The space $L^2_\omega(\Gamma)$ is the set of functions $f$ defined by
\begin{align*}
  L^2_\omega(\Gamma) &= \left\{ f: \Gamma \rightarrow \R \; \big| \; \|f \| < \infty \right\}, &
  \left\| f \right\|^2 &= \left( f, f \right), &
  \left( f, g \right) &= \int_\Gamma f(\bm{x}) g(\bm{x}) \omega(\bm{x}) \dx{x}.
\end{align*}
We assume that the weight function has finite, non-vanishing moments for squared polynomials of all orders, i.e.,
\begin{align*}
  0 < \int_\Gamma \left({x}^\alpha\right)^2 \omega(\bm{x}) &< \infty, & \alpha &\in \mathbb{N}_0.
\end{align*}
The assumption above ensures that polynomials are linearly independent in $L^2_\omega$. Throughout we use $\alpha$ to denote the degree of a polynomial. Given an integrable function $f$, we will also use the notation
\begin{align}\label{eq:I-def}
  I(f) = \int_\Gamma f(x) \omega(x) \dx{x}.
\end{align}
In this paper we seek to construct two sets of $n_1$ and $n_2$ points $\left\{ x^{(q)}_1 \right\}_{q=1}^{n_1} \subset \left\{ x^{(q)}_2 \right\}_{q=1}^{n_2} \subset \Gamma$ with positive weights $w^{(q)}_1,w^{(q)}_2 > 0$ such that
\begin{subequations}\label{eq:quadrature-approx}
\begin{align}
  I(f_1) &= \sum_{q=1}^{n_1} w^{(q)}_1 f_1({x_1^{(q)}}), & f_1 &\in \Pi_{\alpha_1}\\
  I(f_2) &= \sum_{q=1}^{n_2} w^{(q)}_2 f_2({x_2^{(q)}}), & f_2 &\in \Pi_{\alpha_2},
\end{align}
\end{subequations}
where $\Pi_{\alpha}$ is the space of polynomials up to degree $\alpha$:
\begin{align}\label{space_pol}
  \Pi_{\alpha} = \mathrm{span} \left\{ {x}^\alpha \;\; \big| \;\; \alpha \in \mathbb{N}_0 \right\}.
\end{align}
In \eqref{eq:quadrature-approx} we assume $\alpha_1 < \alpha_2$; throughout this paper we use $\alpha_1 , \alpha_2$ to denote the polynomials degrees for $f_1$ and $f_2$.

In many applications, the integrand $f$ in \eqref{eq:I-def} exhibits smoothness (e.g., integrable high-order derivatives), which implies high-order convergence when approximating $f$ by a polynomial. Assuming $f$ is smooth, we expect the quadrature rules in \eqref{eq:quadrature-approx} applied to $f$ to be good approximations to $I(f)$ if $\Pi_{\alpha_1}$ is a large enough subspace. Our main goal in this paper is then to make $\Pi_{\alpha_1}$ and $\Pi_{\alpha_2}$ as large as possible while keeping $n_1,n_2$ as small as possible.

In principle, our numerical method applies to general weight functions in multiple dimensions, but will suffer from the standard complications associated with the curse of dimensionality. In this paper, we restrict our attention and numerical examples to computation of univariate nested quadrature rules for a number of standard weight functions.

\subsection{Gauss Quadrature}

It is well known that the $\omega$-Gauss quadrature rule is optimal quadrature rule for univariate integration, in terms of polynomial accuracy. To define this rule, we first prescribe an orthonormal basis for polynomials. Such a basis of orthonormal polynomials can be constructed via a Gram-Schmidt procedure, with elements $p_n(\cdot)$, where $\deg p_n = n$. Orthonormal polynomials are unique up to a multipicative sign, and satisfy the three-term recurrence relation
\begin{align}\label{eq:three-term-recurrence}
  x p_n(x) = \sqrt{b_n} p_{n-1}(x) + a_n p_n(x) + \sqrt{b_{n+1}} p_{n+1}(x),
\end{align}
for $n \geq 0$, with $p_{-1} \equiv 0$ and $p_0 \equiv 1/\sqrt{b_0}$ to seed the recurrence. The recurrence coefficients are given by
\begin{align*}
  a_n &= (x p_n, p_n), & b^2_n &= \frac{(p_n, p_n)}{(p_{n-1}, p_{n-1})},
\end{align*}
for $n \geq 0$, with $b^2_0 = (p_0, p_0)$. Explicit formulas for the $a_n$ and $b_n$ coefficients are available for various classical orthogonal polynomial families, such as the Legendre and Hermite polynomials \cite{szego_orthogonal_1975}. Gaussian quadrature rules are $n$-point rules that exactly integrate polynomials in $\Pi_{2n-1}$~\cite{Stoer2002,Davis07}, and have essentially complete characterizations.



\begin{theorem}[Gaussian quadrature]\label{TH2_1}
Let $x_{1},\ldots,x_{n}$ be the roots of the $n$th orthogonal polynomial $p_n(x)$ and let $w_{1},\ldots,w_{n}$ be the solution of the system of equations
\begin{equation}\label{THM1_0}
\sum_{q=1}^n p_j(x^{(q)}) w^{(q)} =
\begin{cases}
  \sqrt{b_0}, & \textrm{if } j=0\\
  0, & \textrm{if } j=1,\ldots,n-1.\\
\end{cases}
\end{equation}
Then $y^{(q)} \in \Gamma$ and $v^{(q)}>0$ for $q=1,2,\ldots,n$ and
\begin{equation}\label{THM1_1}
\displaystyle \int_{\Gamma} \omega(x) p(x) dx = \sum_{q=1}^n p(x^{(q)}) w^{(q)}
\end{equation}
holds for all polynomials $p \in \Pi_{2n-1}$.
\end{theorem}
Having knowledge of only a finite number of recurrence coefficients $a_n$, $b_n$, one can compute the Gauss quadrature rule via well established algorithmic strategies in \cite{Golub69,Gautschi68}.

\subsection{Kronrod Nested Rule}
Two different Gauss quadrature sets with e.g. $n$ and $n+1$ nodes can estimate functions with different degrees e.g. $2n-1$ and $2n+1$. The natural question is: what is the maximum degree that the combination of these two sets i.e. $2n+1$ nodes can integrate. Kronrod showed that for the same amount of labour i.e. $2n+1$ nodes in conjunction with $n$-node Gaussian rule one can integrate $3n+1$ ($n$ even) and $3n+2$ ($n$ odd) polynomials.



Given the Gaussian quadrature $A$ with $n$ points Kronrod found quadrature $B$ different than $A$ that has the maximum possible accuracy. To construct quadrature $B$, let $l_n(x)$ denote the degree-$n$ Legendre polynomial, and define $k_{2n+1}(x)=l_n(x) p_{n+1}(x)$ as a polynomial of degree $2n+1$ where $p_n$ is a polynomial of degree $n$.
The polynomial $k_{2n+1}$ is defined such that it is orthogonal to all powers $x^i$ for $i=0,\ldots,n$. 
Let the roots of $k_{2n+1}$ be $x_1,\ldots,x_{2n+1}$. The $(2n+1)$-point interpolatory quadrature rule of the form
\begin{equation*}
B(f) =  \sum_{i=1}^{2n+1} w_i f(x_i) dx,
\end{equation*}
has an accuracy no less than $2n$. However, the accuracy of this rule is substantially larger than $2 n$.
\begin{theorem}[Kronrod Rule~\cite{Kronrod64}]\label{TH2_1}
Quadrature $B(f)$ has an accuracy of $3n+1$ for even $n$ and of $3n+2$ for odd $n$.
\end{theorem}
Kronrod found these quadrature sets for Legendre polynomials up to $40$ points. Our numerical method finds nested quadrature rules for arbitrary high order and various weight functions.

\subsection{Nested Rule via Numerical Optimization}\label{sec:nested_rule}

The Kronrod rule revolves around the ability to find the root of the polynomial $k_{2n+1}$ and as mentioned it has been proven for certain number of nodes and polynomial degrees and weight function. Similarly to the Kronrod rule we construct a nested rule which centers around the direct moment-matching conditions for the lower and the upper rule. The following result states the essence of our nested numerical quadrature:

\begin{proposition}\label{PR2_1}
  Let $n_1, n_2, \alpha_1, \alpha_2 \in \N$ be given, with $n_1 < n_2$ and $\alpha_1 < \alpha_2$. 
  Suppose that $\left\{ x^{(q)}_1 \right\}_{q=1}^{n_1} \subset \left\{ x^{(q)}_2 \right\}_{q=1}^{n_2}$  and $\left\{ w^{(q)}_1 \right\}_{q=1}^{n_1}$, $\left\{ w^{(q)}_2 \right\}_{q=1}^{n_2}$ are the solution of two systems of equations
\begin{equation}\label{PRO2_0}
\begin{array}{l }
\displaystyle \sum_{q=1}^{n_1} {p}_j(x_1^{(q)}) w_1^{(q)} =
\begin{cases}
  \sqrt{b_0}, &\textrm{if } j= 0\\
  0, &\textrm{if } j=1,\ldots,\alpha_1\\
\end{cases}
\\
\displaystyle \sum_{q=1}^{n_2} {p}_{j}(x_2^{(q)}) w_2^{(q)} =
\begin{cases}
  \sqrt{b_0}, &\textrm{if } j= 0\\
  0, &\textrm{if }  j=1,\ldots,\alpha_2\\
\end{cases}
\end{array}
\end{equation}
then
\begin{equation}\label{PRO2_1}
\displaystyle \int_{\bm \Gamma} \omega( x) p_1(x) d x = \sum_{q=1}^{n_1} p_1(x_1^{(q)}) w_1^{(q)}, \quad \displaystyle \int_{\bm \Gamma} \omega(x) p_2(x) d x = \sum_{q=1}^{n_2} p_2( x_2^{(q)}) w_2^{(q)}
\end{equation}
holds for all polynomials $p_1 \in {\Pi}_{\alpha_1},p_2 \in {\Pi}_{\alpha_2}$.
\end{proposition}

We note that that $\int_\Gamma {p}_{j}({x}) \omega(x) dx=0$ when $j \neq 0$ due to orthogonality, and hence Equations~\eqref{PRO2_0} are simply moment-matching conditions over two different polynomial spaces ${\Pi}_{\alpha_1}$ and ${\Pi}_{\alpha_2}$. The existence of quadrature rules satisfying the above conditions does not guarantee the positivity of weights nor it ensures that the nodes lie in $\Gamma$. We enforce these conditions in our numerical method, which is explained in the next section. We also provide a general guideline for the choices of $n_1,n_2$ and $\alpha_1,\alpha_2$.

We finally note that Proposition \ref{PR2_1} can be generalized for multi-dimensional polynomials however the relationship between $n_1,n_2,\alpha_1,\alpha_2$ is less obvious, and the computational cost becomes prohibitive. Hence we focus onlly on computation of univariate nested rules in this paper. Next, we briefly discuss the multivariate construction of quadrature rules based on the Smolyak algorithm which is known to enjoy efficiency gains from using univariate nested rules.

In the next section we provide a computational framework that solves a constrained version of \eqref{PRO2_0} via optimization. 

\section{Numerical Method}\label{S3}

We aim to compute nodes $\bm{x_1} = \left\{ {x_1^{(1)}}, \ldots, {x_1^{(n_1)}} \right\} \subset \bm{x_2} = \left\{ {x_2^{(1)}}, \ldots, {x_2^{(n_2)}} \right\} \in \Gamma$
and positive weights $\bm{w_1}=\left\{ {w_1^{(1)}},\ldots, {w_1^{(n_1)}}  \right\}, \bm{w_2}=\left\{ {w_2^{(1)}},\ldots, {w_2^{(n_2)}}  \right\} \in (0, \infty)$ that satisfies \eqref{PRO2_0} up to a tolerance of a prescribed $\epsilon > 0$. We directly formulate \eqref{PRO2_0} as
\begin{equation}\label{S3_1_0}
\begin{array}{l}
  \bm R_1(\bm d_1) = \bm V_{\alpha_1}(\bm x_1) \bm w_1 - \sqrt{b_0} \bm{e}_{1,\alpha_1+1} = \bm 0,  \\
  \\
  \bm R_2(\bm d_2) = \bm V_{\alpha_2}(\bm x_2) \bm w_2 - \sqrt{b_0} \bm{e}_{1,\alpha_2+1} = \bm 0,  \\
  \\
  \bm x_1 \subset \bm x_2 \in \Gamma,\\
  \\
  \bm w_1,\bm w_2 > \bm 0,
\end{array}
\end{equation}
where $\bm V$ denotes the Vandermonde matrix, e.g. $\bm V_{\alpha_1}$ is the Vandermonde matrix that consists of polynomial evaluations up to degree $\alpha_1$, $\bm{e}_{1,\alpha_i+1}$ is the first unit vector with size $\alpha_i+1$ and $\bm d_1 = (\bm x_1, \bm w_1), \bm d_2 = (\bm x_2, \bm w_2)$ are decision variables. Instead of solving this constrained root finding problem, we introduce a closely related constrained minimization problem on decision variables $\bm d=(\bm x_2,\bm w)$:

\begin{equation}\label{S3_3_3}
\begin{array}{r l l}
  \displaystyle \mathop{\min}_{\bm x_2, \bm w} & \displaystyle ||\bm R||_2  &  \\
  \text{subject to} & \bm x_2 \in \Gamma,\\
                    & \bm w> \bm 0. &
\end{array}
\end{equation}
Above we have eliminated the decision variable $\bm{x}_1$, since $\bm{x}_1 \subset \bm{x}_2$. We have also introduced the vectors $\bm{R}$ and $\bm{w}$, defined as
\begin{equation}\label{eq:R-def}
\begin{array}{l l }
 {\bm{R}}=  \left[
    \begin{array}{c}
      \bm{R_1} \\
      \bm{R_2}
    \end{array}
  \right],
  &
  {\bm{w}}=  \left[
    \begin{array}{c}
      \bm{w_1} \\
      \bm{w_2}
    \end{array}
  \right].
  \end{array}
\end{equation}
The total number of decision variables above is $n_1 + 2 n_2$. Ideally we need the solution to \eqref{S3_1_0}, which also solves \eqref{S3_3_3}, but the reverse is not necessarily true. In practice we fix the polynomial degree $\alpha_2$, solve \eqref{S3_3_3}, and when the solution exhibits nonzero values of $\|\bm{R}\|$, we decrease $\alpha_2$ and repeat. Using this strategy, we empirically find that we can satisfy $\|\bm{R}\| \leq \epsilon$ in all situations we have tried i.e. we have been able to re-generate available Kronrod rules with very small tolerances, and we have been able to find new nested quadrature rules.

Our approach therefore effectively solves \eqref{S3_1_0} via repeated applications of \eqref{S3_3_3}. Our numerical approach to solve \eqref{S3_3_3} is a modification of the algorithm in \cite{Keshavarzzadeh_DQ2017}. We describe in brief the ingredients of this algorithm, and in more detail the portions that are specific to our construction. The overall algorithm has four major steps, each of which are described in the subsequent sections:
\begin{enumerate}
  \item Section \ref{S3_3}  Penalization: transforming constrained root finding into unconstrained minimization problem by augmenting the objective with penalty functions
 \item Section \ref{sec:method-newton}  Iteration: Gauss-Newton algorithm for unconstrained minimization
  \item Section \ref{S3_4}  Regularization: numerical regularization to address ill-conditioned Gauss-Newton update steps
  \item Section \ref{S3_2}  Initialization: specification of an initial guess
\end{enumerate}

\subsection{Penalty Method}\label{S3_3}

We use penalty methods to solve the constrained optimization problem~\eqref{S3_3_3} by adding a high cost for violated constraints to the objective function. We subsequently solve an unconstrained minimization problem on the augmented objective.

We choose a popular penalty function, the non-negative and smooth quadratic function. Taking as an example $\Gamma = [-1,1]$, the constraints and corresponding penalties $P_j$, $j=1, \ldots, 2n_2+n_1$ as a function of the $2n_2+n_1$ decision variables $\bm{d} = \left(\bm{x_2}, \bm{w}\right)$ can be expressed as
\begin{align*}
  -1 \leq x^{(j)}_2 \leq 1 \hskip 5pt &\Longrightarrow \hskip 5pt P_j\left(\bm{d}\right) = \left(\max[0,x_j-1,-1-x_j]\right)^2, j = 1, \ldots, n_2\\
  w^{(j)}_2 \geq 0 \hskip 5pt &\Longrightarrow \hskip 5pt P_{n_2+j}\left(\bm{d}\right) = \left(\max[0,-w_j]\right)^2, j = 1, \ldots, n_2\\
  w^{(j)}_1 \geq 0 \hskip 5pt &\Longrightarrow \hskip 5pt P_{2n_2+j}\left(\bm{d}\right) = \left(\max[0,-w_j]\right)^2, j = 1, \ldots, n_1.
\end{align*}
The total penalty in this case is then expressed as
\begin{align*}
  P^2\left(\bm{d}\right) = \sum_{j=1}^{2n_2+n_1} P^2_j\left(\bm{d}\right).
\end{align*}

The penalty approach solves the constrained problem \eqref{S3_3_3} by using a sequence of unconstrained problems indexed by $k \in \N$ with objective functions
\begin{align}\label{eq:g-def}
  g\left(c_k, \bm{d} \right) = \left\|\widetilde{\bm{R}}_k \right\|^2_2 = \left\| \bm{R}\right\|^2_2 + c^2_k P^2\left(\bm{d}\right),
\end{align}
where
\begin{align*}
  \widetilde{\bm{R}}_k &= \left[
    \begin{array}{c}
      \bm{R} \\
      c_k P_1 \\
      c_k P_2 \\
      \vdots \\
      c_k P_{2n_2+n_1}
    \end{array}
  \right].
\end{align*}
The positive constants $c_k$ are monotonically increasing with $k$, i.e., $c_{k+1}> c_k$. Each unconstrained optimization yields an updated solution point $\bm d^{k}$, and as $c_k \rightarrow \infty$ the solution point of the unconstrained problem will converge to the solution of constrained problem.

We now formulate an unconstrained minimization problem with sequence of increasing $c_k$ associated with the objectives $g$ in \eqref{eq:g-def} for the decision variables $\bm{d} = (\bm{x_2}, \bm{w})$,
\begin{align}\label{eq:unconstrained-g}
  \displaystyle \mathop{\min}_{\bm{d}} \displaystyle g(c_k, \bm{d})
\end{align}
which provides an approximation to the solution of the constrained root-finding problem \eqref{S3_1_0}. We specify the constants $c_k$ as follows: if $\bm{d}$ is the current iterate for the decision variables, we use the formula
\begin{align*}
  c_k = \max \left\{A, \frac{1}{||\bm{R}(\bm d)||_2} \right\},
\end{align*}
where $A$ is a tunable parameter that is meant to be large and set to $A = 10^3$ in our numerical examples. We also note that we never have $c_k = \infty$ so that our iterations cannot exactly constrain the computed solution to lie in the feasible set. In practice we reformulate constraints to have non-zero penalty within a small radius inside the feasible set. For example, instead of enforcing $w_j > 0$, we enforce $w_j > 10^{-6}$.


\subsection{The Gauss-Newton Minimization}\label{sec:method-newton}
Now that the constrained problem \eqref{S3_3_3} is transformed to a sequence of unconstrained problems \eqref{eq:unconstrained-g}, we can use standard unconstrained optimization tools such as gradient descent or Newton's method.

The gradient-based approaches require the Jacobian of the objective function with respect to the decision variables. We define

\begin{align}\label{S3_3_6}
  \bm{J}(\bm{d}) &= \left[
  \begin{array}{c}
      \bm J_1  \\
      \bm J_2
          \end{array}
     \right]  =\left[
  \begin{array}{c}
      \partial \bm R_1 /\partial \bm d \\
      \partial \bm R_2/ \partial \bm d
    \end{array}
     \right] \in \R^{(\alpha_1+\alpha_2+2)\times (2n_2+n_1)  },
  &
  \widetilde{\bm{J}}_k = \pfpx{\widetilde{\bm{R}}_k}{\bm{d}} &= \left[
    \begin{array}{c}
      \bm{J} \\ 
      c_k \partial P_1 /\partial \bm d \\
      c_k \partial P_2 / \partial \bm d  \\
      \vdots \\
      c_k \partial P_{2n_2+n_1} / \partial \bm d
    \end{array}
  \right],
\end{align}
where $\pfpx{P_j}{\bm{d}} \in \R^{1 \times (2n_2+n_1)}$ is the Jacobian of $P_j$ with respect to the decision variables. We note that the Lipschitz continuity as well as easy evaluation of Jacobians are ensured using the quadratic penalty function. For instance, the first part of $\bm{J}$ has entries
\begin{align}\label{eq:jacobian-entries}
  (J_1)_{m, i} &= \pfpx{{p_1}_{m}\left(x_1^{(i)}\right)}{x_1^{(i)}} w_i,  &
  (J_1)_{m, n_1 + i} &= {p_1}_{m}({x_1}^{(i)}),
\end{align}
for $m = 0, \ldots, \alpha_1$, $i = 1, \ldots, n_1$ where we define ${p_1}$ as in Section~\ref{sec:nested_rule}. The formula is similar for entries of $\bm{J}_2$. Computing entries of the Jacobian matrix $\bm{J}$ is straightforward since we only need to compute derivatives of univariate polynomials. A manipulation of the three-term recurrence relation \eqref{eq:three-term-recurrence} yields the recurrence
\begin{align*}
  \sqrt{b_{m+1}} p_{m+1}'(x) = (x - a_m) p_m'(x) - \sqrt{b}_m p_{m-1}'(x) + p_m(x),
\end{align*}
which is used to evaluate the partial derivatives in $\bm{J}$.

We use the same index iterations $k$ as that defining the sequence of unconstrained problems \eqref{eq:unconstrained-g}; Thus, our choice of $c_k$ changes at each iteration. A simple gradient descent is based on the update with the form
\begin{align*}
  \bm d^{k+1} &= \bm d^{k} - \alpha \pfpx{\|\widetilde{\bm{R}}_k\|_2}{\bm{d}}, & \pfpx{\|\widetilde{\bm{R}}_k\|_2}{\bm{d}} = \frac{\widetilde{\bm{J}}_k^T \widetilde{\bm{R}}}{\|\widetilde{\bm{R}}_k\|_2},
\end{align*}
with $\alpha$ a customizable step length that is frequently optimized via, e.g., a line-search algorithm. In contrast, a variant of Newton's method applied to a rectangular systems is the Gauss-Newton method \cite{Stoer2002}, with update iteration
\begin{align}\label{eq:Gauss-newton}
  \bm d^{k+1} &= \bm d^{k} - \Delta \bm{d}, & \Delta \bm{d} &= \left( \widetilde{\bm{J}}_k^T \widetilde{\bm{J}}_k \right)^{-1} \widetilde{\bm{J}}_k^T \widetilde{\bm{R}}_k,
\end{align}
where both $\widetilde{\bm{J}}_k$ and $\widetilde{\bm{R}}_k$ are evaluated at $\bm{d}^k$. The iteration above reduces to the standard Newton's method when the system is square, i.e., $\alpha_1+\alpha_2+2 = 2n_2+n_1$. It is well known that Newton's method converges quadratically to a local solution for a sufficiently close initial guess $\bm{d}^0$ versus the gradient descent which has linear convergence \cite{Bertsekas08}. We use Gauss-Newton iterations that we find robust for our numerical algorithm.

Starting with an initial guess $\bm{d}^0$, we repeatedly apply the Gauss-Newton iteration \eqref{eq:Gauss-newton} until a stopping criterion is met. We terminate our iterations when the residual norm falls below a user-defined threshold $\epsilon$, i.e., when $||\widetilde{\bm R}||_2 < \epsilon$.

We also define another useful quantity to monitor during the iteration process which is the magnitude of the Newton decrement. This measure often reflects quantitative proximity to the optimal point \cite{Boyd04}. Based on its definition,
the Newton decrement is the norm of the Newton step in the quadratic norm defined by the Hessian. I.e., the Newton decrement norm for a function $f(\bm{x})$, is $|| \Delta \bm d||_{\nabla^2 f(\bm x)} = (\Delta \bm d^T \nabla^2 f(\bm x) \Delta \bm d  )^{1/2}$, where $\nabla^2 f$ is the Hessian of $f$. In our minimization procedure with non-squared systems we define
\begin{equation}\label{S3_1_6}
  \eta  = \big(\Delta \bm d^T (\widetilde{\bm{J}}_k^T \widetilde{\bm{R}}_k)\big)^{1/2},
\end{equation}
as a surrogate for a Hessian-based Newton decrement, which decreases as $\bm d \rightarrow \bm d^{*}$.

Finally we note that, for a given nested quadrature rule with $(n_1,n_2,\alpha_1,\alpha_2)$ we cannot guarantee that a solution to \eqref{S3_1_0} exists. In this case our Gauss-Newton iterations will exhibit residual norms stagnating at some positive value while the Newton decrement is almost zero. When this occurs, we either re-initialize our decision variables or decrease the magnitude of the higher order $\alpha_2$ and continue the optimization procedure. By gradually decreasing $\alpha_2$, we find a successful combination of $(n_1,n_2,\alpha_1,\alpha_2)$ that meet the residual tolerance criterion.

\subsection{Tikhonov Regularization}\label{S3_4}
The evaluation of the Newton update~\eqref{eq:Gauss-newton} is a critical part of our scheme. In our rectangular system, the update is the least-squares solution $\Delta \bm{d}$ to the linear system
\begin{align*}
  \widetilde{\bm{J}} \Delta \bm{d} = \widetilde{\bm{R}},
\end{align*}
where $\widetilde{\bm{J}} = \widetilde{\bm{J}}_k\left(\bm{d}^k\right)$, and $\widetilde{\bm{R}} = \widetilde{\bm{R}}_k\left(\bm{d}^k\right)$ and in this section we omit explicit notational dependence on the iteration index $k$. Based on our numerical experience, the Jacobian  $\widetilde{\bm{J}}$ can be ill-conditioned. Therefore, to effectively solve the above least-squares problem we consider a generic regularization of the equality:
\begin{equation}\label{S3_4_1}
  \displaystyle \mathop{ \textrm{minimize}}_{\Delta \bm d} \quad ||\widetilde{\bm{J}} \Delta \bm d - \widetilde{\bm{R}} ||_p  \quad \textrm{subject to} \quad  || \Delta \bm d||_q < \tau,
\end{equation}
where $p$, $q$, and $\tau$ are free parameters. The Pareto curve that characterizes the trade off between the objective norm and solution norm is shown to be convex in~\cite{Vandenberg08,Vandenberg11} for generic norms $1 \leq (p,q) < \infty$. In this paper we utilize a 2-norm regularization i.e. $p=q=2$ which can be solved via an approach in \cite{golub_matrix_1996}; however, this procedure provides no clear guideline on choosing $\tau$. Thus we adopt an alternative approach based on Tikhonov regularization. This approach is a penalized version of the $p=q=2$ optimization \eqref{S3_4_1}:
\begin{equation}\label{S3_4_2}
\Delta \bm d_{\lambda} = \textrm{argmin} \Big \{ ||\widetilde{\bm{J}} \Delta \bm d - \widetilde{\bm{R}} ||^2_2 + \lambda ||\Delta \bm d||_2^2  \Big \},
\end{equation}
where $\lambda$ is a regularization parameter that can be selected by the user. However, this parameter can significantly impact the quality of the solution with respect to the original least-squares problem. Assuming that a value for $\lambda$ is prescribed, the solution to \eqref{S3_4_2} can be obtained via the singular value decomposition (SVD) of $\widetilde{\bm{J}}$. The SVD of matrix  $\widetilde{\bm{J}}_{\mathrm{N} \times \mathrm{M}}$ (for $\mathrm{N}<\mathrm{M}$) is given by
\begin{equation}
  \widetilde{\bm{J}} = \displaystyle \sum_{i=1}^{\mathrm{N}} \bm u_i \sigma_i \bm v^T_i,
\end{equation}
where $\sigma_i$ are singular values (in decreasing order), and $\bm{u}_i$ and $\bm{v}_k$ are the corresponding left- and right-singular vectors, respectively. The solution $\Delta \bm d_{\lambda}$ is then obtained as
\begin{equation}\label{S3_4_5}
  \Delta \bm d_{\lambda} = \displaystyle \sum_{i=1}^{\mathrm{N}} \rho_i   \displaystyle \frac{ \bm u^T_i \widetilde{\bm{R}}}{ \sigma_i} \bm v_i,
\end{equation}
where $\rho_i$ are Tikhonov filter factors denoted by
\begin{equation}\label{S3_4_4}
\begin{array}{l }
{\rho_i} =  \displaystyle \frac{\sigma_i^2}{\sigma_i^2+\lambda^2} \simeq \begin{cases}
1 \hspace{1.25cm} \sigma_i \gg \lambda,  \\
\sigma_i^2/\lambda^2 \hspace{0.5cm}   \sigma_i \ll \lambda.
\end{cases}
\end{array}
\end{equation}
Tikhonov regularization filters singular values that are below the threshold $\lambda$. Therefore a suitable $\lambda$ is bounded by the extremal singular values of $\widetilde{\bm{J}}$. One approach to select the regularization parameter is to  analyze the ``$L$-curve" of singular values \cite{Hansen98,Hansen93} and choose $\lambda$ that correponds to the corner of $L$-curve that has the maximum curvature. Approximation of the curvature with respect to singular value index can be used to find the index with maximum curvature, and the singular value corresponding to this index prescribes $\lambda$.

In practice, we evaluate the curvature of the singular value spectrum via finite differences on $\log(\sigma_i)$ (where the singular values are directly computed) and choose the singular value that corresponds to the first spike in the spectrum. The regularization parameter can be updated after several, e.g., $30-50$, Gauss-Newton iterations. In our small size, i.e., univariate, problems a small number of fixed appropriate $\lambda$ throughout the Gauss-Newton scheme also yields the desirable solutions.

In our experiments we also find that adding a regularization parameter to all singular values and computing the regularized Newton step as $\Delta \bm d_{\lambda} = \sum_{i=1}^{\mathrm{N}}  [({ \bm u^T_i \widetilde{\bm{R}}})/({ \sigma_i + \lambda})] \bm v_i$ enhances the convergence when $\bm d $ is close to the root i.e. $||\widetilde{\bm{R}}||$ is sufficiently small.

\subsection{Initialization}\label{S3_2}

The first step of the algorithm requires an initial guess $\bm{d}^0$ for nodes and weights. Mimicking the Kronrod approach, we choose $n_1$ for the number of nested nodes and $n_2=2n_1+1$ for the number of main nodes. We also fix the order for the nested rule based on Gauss quadrature (similarly to Kronrod rule) i.e. $\alpha_1=2n_1-1$. The order for the main rule $\alpha_2$ is varied until a desired residual norm is achieved. We initialize the main and nested nodes in an interlacing order as depicted in the Figure~\ref{fig_init}.

\begin{figure}[h]
\includegraphics[width=6in]{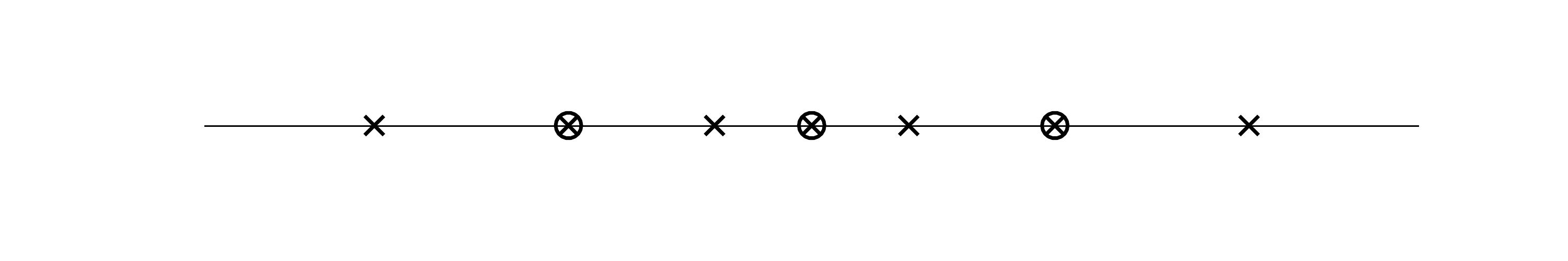}\\
\caption{\small{Configuration of initial nested and main nodes for $n_1=3,~n_2=2n_1+1=7$. The nested nodes $\bm{x}_1$ are shown with circumscribed circles ($\circ$), whereas the main nodes $\bm{x}_2$ are marked with $\times$.}}\label{fig_init}
\end{figure}

We consider $\omega$ a probability density in all cases (i.e., $b_0 = 1$) so that the weights for main and nested rules are normalized to yield $\sum_i w^{(i)}=1$. We initialize nodes based on the Gauss quadrature nodes for the main rule. Once the main rule is initialized the nested nodes are readily available as shown in Figure~\ref{fig_init}. The weights for main and nested rules are simply considered as uniform values i.e. $\bm w_1 = {1}/{n_1}, \bm w_2 = {1}/{n_2}$.

\begin{remark}
For unbounded weights e.g. Gaussian weight function the optimal main rule has an order much smaller than the optimal Gauss quadrature order. That means the nodes are located in a much shorter range compared to original Gauss quadrature. To incorporate this effect we use the ratio between the maximum value of the Gaussian quadrature with $\alpha_2$ degree i.e. $\lfloor {(\alpha_2+1)/2} \rfloor$ nodes where $\lfloor . \rfloor $ is the floor function and the Gaussian quadrature with $n_2$ nodes. We use this ratio to shrink the range of Gaussian quadrature with $n_2$ nodes.
\end{remark}

Algorithm \ref{alg:quad-nested} summarizes our numerical method for nested quadrature, including the steps in Sections \ref{S3_3}--\ref{S3_2}.

\begin{algorithm}
\caption{Nested Quadrature Generation}\label{alg:quad-nested}
\label{alg1}
\begin{algorithmic}[1]
\State Initialize nodes and weights, $n_1$ for the nested nodes and $n_2=2n_1+1$ for the main nodes associated with the given orthogonal polynomial system cf. Section~\ref{S3_2}.
\State Specify the residual tolerance e.g. $\epsilon=10^{-12}$
\State Set $\alpha_2^{*}=0$
\While{ $||\widetilde{\bm{R}}||> \epsilon$}
\State Compute $\widetilde{\bm{R}}$ and $\widetilde{\bm{J}}$ using \eqref{eq:g-def}, \eqref{S3_1_0}, \eqref{S3_3_6}, and \eqref{eq:jacobian-entries}.
\State Determine the Tikhonov parameter $\lambda$ from the SVD of $\widetilde{\bm{J}}$
\State Compute the Newton step $\Delta \bm d$ from \eqref{S3_4_5}
\State Update the decision variables $\bm d^{k+1} = \bm d^{k} - \Delta \bm d$.
\State Compute the residual norm $||\widetilde{\bm{R}}||_2$ and Newton decrement $\eta$ from \eqref{S3_1_6}.
\If {$\eta < \epsilon$ and $||\widetilde{\bm{R}}||_2 \gg \epsilon$}
\State $\alpha_2 \gets \alpha_2-1$ and go to line $4$.
\EndIf
\EndWhile
\If {$\alpha_2 = \alpha_2^{*}$}
\State Return
\Else
\State $\alpha_2^{*} \gets \alpha_2$, $\alpha_2 \gets \alpha_2+1$ and go to line $4$
\EndIf
\end{algorithmic}
\end{algorithm}

\section{Numerical Examples}\label{sec:results}

\subsection{Univariate examples}
In this section we investigate the effectiveness of the nested quadrature rules on evaluating univariate integrals. Some of these experiments serve as validation where we verify existing nested rules. The remaining experiments focus on highlighting the computation of new rules using the general methodology we have introduced.

In addition to quadrature corresponding to Kronrod rules, we have also generated existing and new quadrature for Gauss-Kronrod-Patterson rules which we briefly describe at this juncture.

Patterson~\cite{Patterson68} extended Kronrod rule by generating nested sequences of univariate quadrature rules. In other words, a Kronrod-Patterson rule with accuracy level $i$ adds a number of points to the preceding level set $\mathbb{X}_{i}$ sequentially and updates the weights accordingly. As such this procedure results in $\mathbb{X}_{i} \subset \mathbb{X}_{i+1} \subset  \ldots$. These nodes have been specifically obtained for Legendre polynomials~\cite{Patterson68} and later in a similar fashion for Hermite polynomials~\cite{Genz96}. To investigate these rules we slightly amend our algorithm and generate some of these rules in addition to sequential nested rules for Chebyshev polynomials, which to our knowledge are new.

To generate Kronrod-Patterson sequential nested rules, we use an initial guess with fixed $n_1$ number of points and optimize $n_1+1$ additional points. The sequence starts with $n_1=1$ and $x_1^{(1)}=0$ which integrates $\alpha_1=1$ order. The next step is to add $1+1=2$ points and optimize their locations and weights (while the initial $1$ point is fixed) to achieve an optimal order e.g. $\alpha_2=5$ in the case of Legendre and Chebyshev. At this point the new rule with $n_2=3$ and $\alpha_2=5$ is considered as the nested rule for the next step where $3+1=4$ points are added to find a new optimal $7$-point rule. This procedure is continued for larger number of points.

It is noted that we only consider $\bm R_2 $ residual and its associated constraints' residual (since $\bm R_1 = \bm 0$ already) at each sequential optimization and compute $\bm R_2$ as a function of $(\bm x_2,\bm w_2)$. However, the Jacobian only includes the derivatives of $\bm R_2$ with respect to additional $\bm x_2\backslash\bm x_1$ nodes and the weights $\bm w_2$. As such the resulting Newton update has $n_2-n_1+n_2 = 3n_1+2$ entries. These values are concatenated to $n_1$ zeros to update the current iteration $(\bm x_2,\bm w_2)$.


\subsubsection{Verification of Existing Rules: Bounded Domains}\label{existing_rules_BD}
In our numerical experiments for integration on bounded domains we always find the Kronrod relationship given $n_1$, i.e.,
\begin{align}\label{eq:kronrod-relationship}
  n_2 &=2n_1+1, & \alpha_1 &= 2n_1-1, & \alpha_2 &= \left\{\begin{array}{rl} 3n_1+1, & \textrm{$n$ even} \\
                                                                             3n_1+2, & \textrm{$n$ odd}
\end{array}\right.
\end{align}
for all tested values of $n_1$.  We will consider three different polynomial families: Legendre, Chebyshev and asymmetric Jacobi. We report existing results on Legendre Polynomial in this section and show the results for Chebyshev and Jacobi in section~\ref{new_rules_BD}. Figure~\ref{leg15} shows the initial and final nodes and weights for the Legendre polynomial associated with $n_1=7, n_2=15, \alpha_1=13, \alpha_2=23$.

\begin{figure}[h]
\centering
\includegraphics[width=6in]{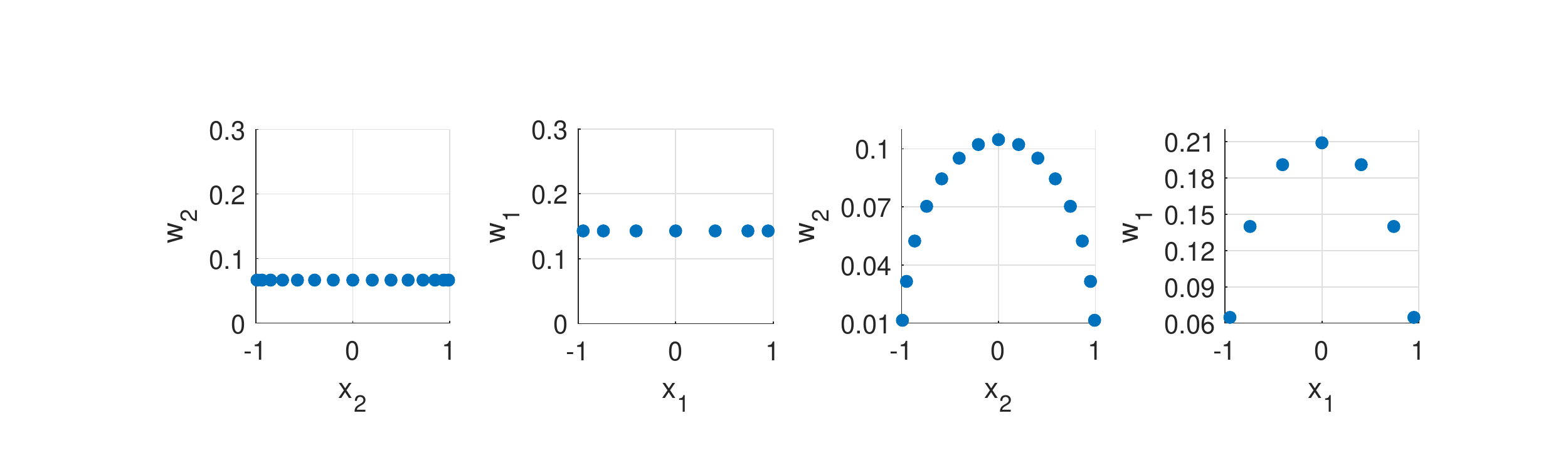}\\
\caption{\small{Initial nodes and weights (two left plots), final nodes and weights (two right plots) for Legendre polynomial associated with  $n_1=7, n_2=15, \alpha_1=13, \alpha_2=23$.}}\label{leg15}
\end{figure}

Figure~\ref{leg-patt} shows the sequence of nested quadrature for Gauss-Kronrod-Patterson rules. These rules are associated with $n_2=3,7,15,31,63$ and $\alpha_2=5,11,23,47,95$. Our algorithm found a lower order ($\alpha_2=91$ instead of $\alpha_2=95$) for the $63$-point rule which we report in the next section however we were able to recover these existing points when we used an initial guess close to the optimal points reported in~\cite{Qsparse}.
\begin{figure}[h]
\centering
\includegraphics[width=6in]{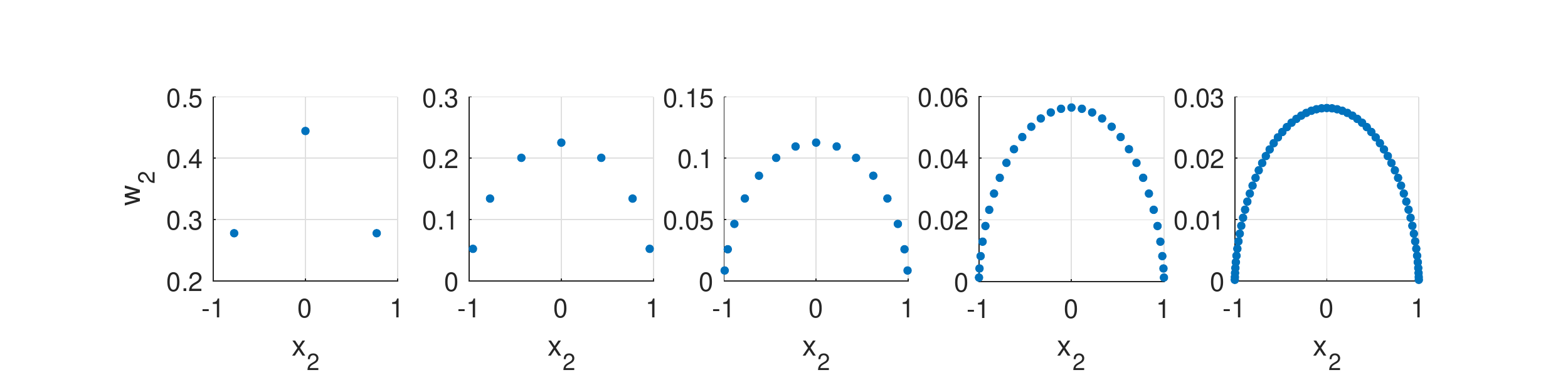}\\
\caption{\small{Optimized nested sequence of quadrature rules $n_2=3,7,15,31,63$ for Legendre polynomial with orders $\alpha_2=5,11,23,47,95$.}}\label{leg-patt}
\end{figure}

\subsubsection{Verification of Existing Rules: Unbounded Domains}\label{existing_rules_UBD}

To generate existing rules for unbounded domains we consider the sequence of nested rules associated with the Gaussian weight function $\omega(x) \varpropto e^{-x^2}$. We generate $n_2=3,9,19,35$ nodes corresponding to polynomial orders $\alpha_2=5,15,29,51$  for tolerance $\epsilon=10^{-12}$. For $n_2=19,35$ we used initial guesses close to those reported in~\cite{Qsparse} and were able to recover similar nodes cf. Figure~\ref{herm-patt}. The computed $19$-point rule has two (symmetric) nodes with negative weights. To recover this rule we relaxed our optimization with no constraint on positive weights. This rule is the only quadrature with negative weights throughout this paper and we generated it to solely verify the existing rule.

\begin{figure}[h]
\centering
\includegraphics[width=6in]{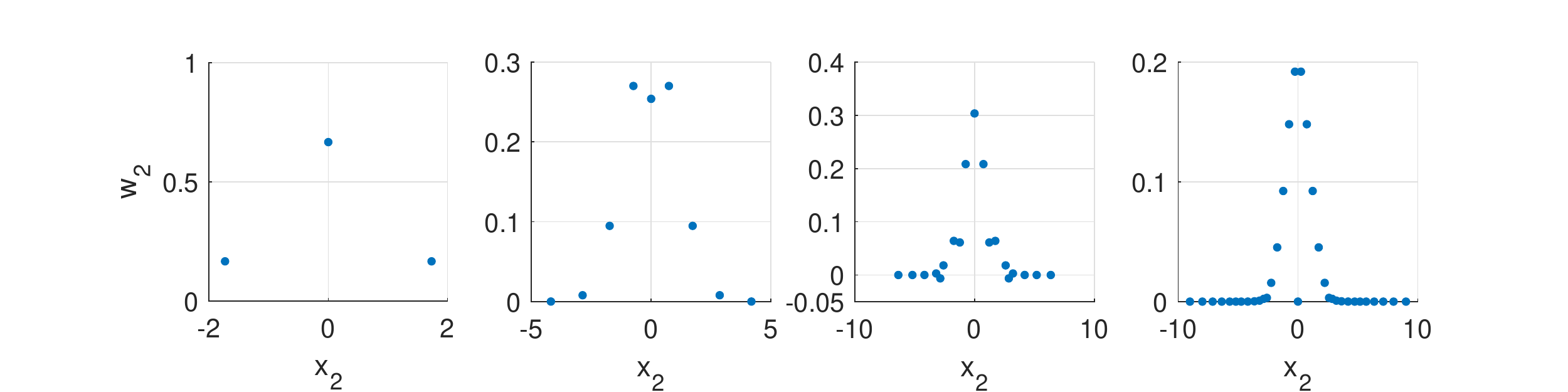}\\
\caption{\small{Optimized nested sequence of quadrature rules $n_2=3,9,19,35$ for weight function $\omega(x) \varpropto e^{-x^2}$.}}\label{herm-patt}
\end{figure}

\subsubsection{Generation of New Rules: Bounded Domains}\label{new_rules_BD}

The optimal quadrature for high polynomial order (and high number of nodes) i.e. $n_1=100,n_2=201, \alpha_1=199,\alpha_2=301$ for both cases of Legendre and Chebyshev is shown in Figure~\ref{leg-cheb}. The MATLAB scheme takes $173$ and $261$ iterations which amount to $36.17~sec$ and $53.18~sec$ respectively for these cases on a personal desktop computer with Intel Core $i7-5930K~ @3.5~ GHz$ CPU.

\begin{figure}[h]
\centering
\includegraphics[width=5in]{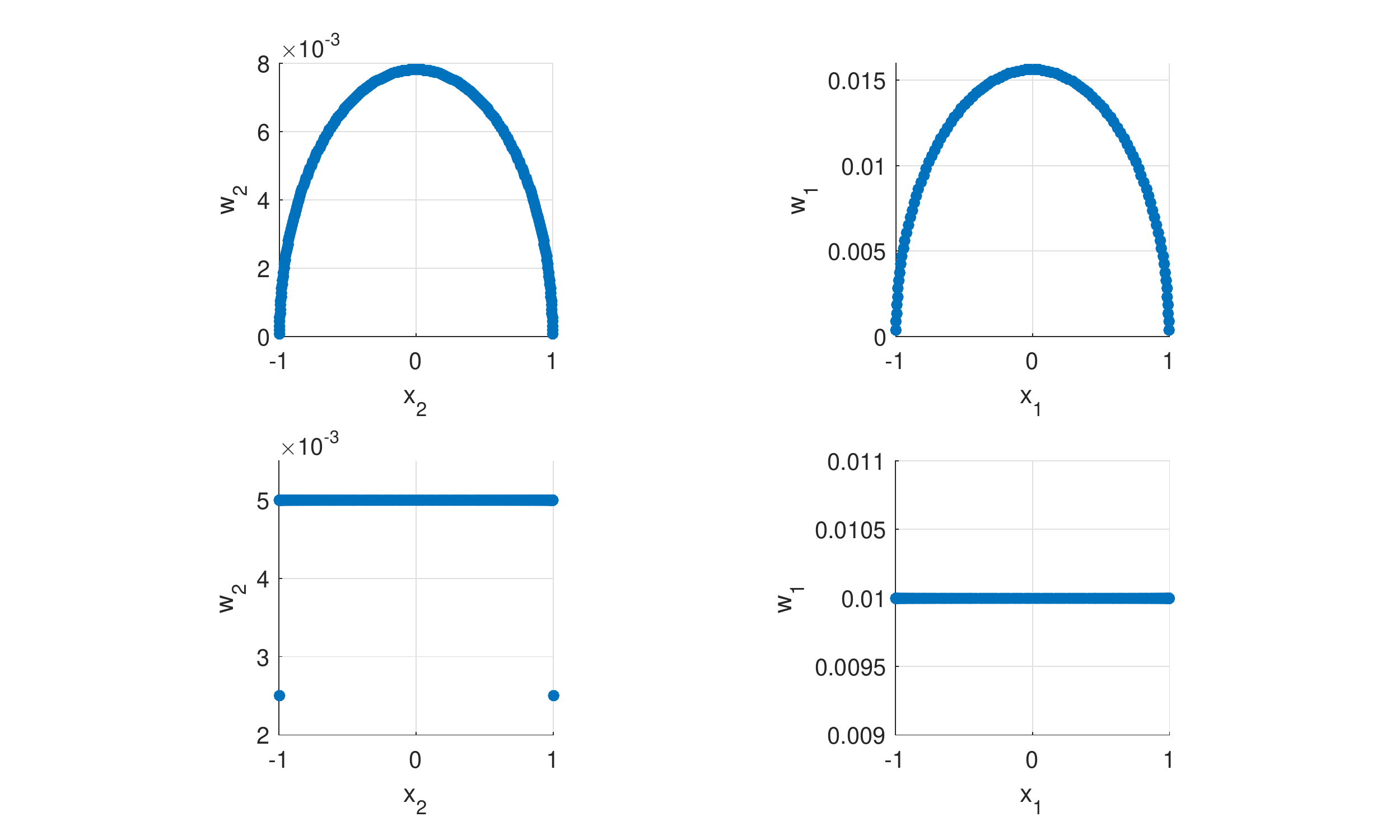}\\
\caption{\small{Optimal main and nested nodes and weights for Legendre (top) and Chebyshev (bottom) polynomials associated with  $n_1=100, n_2=201, \alpha_1=199, \alpha_2=301$.}}\label{leg-cheb}
\end{figure}

The experiment for asymmetric Jacobi is associated with $\alpha=0, \beta=0.3$. Figure~\ref{jac} shows the optimal nodes and weights for $n_1=10$ and $n_1=100$. Again we emphasize that we find the Kronrod rule cf. Equation~\eqref{eq:kronrod-relationship}.


\begin{figure}[h]
\centering
\includegraphics[width=6in]{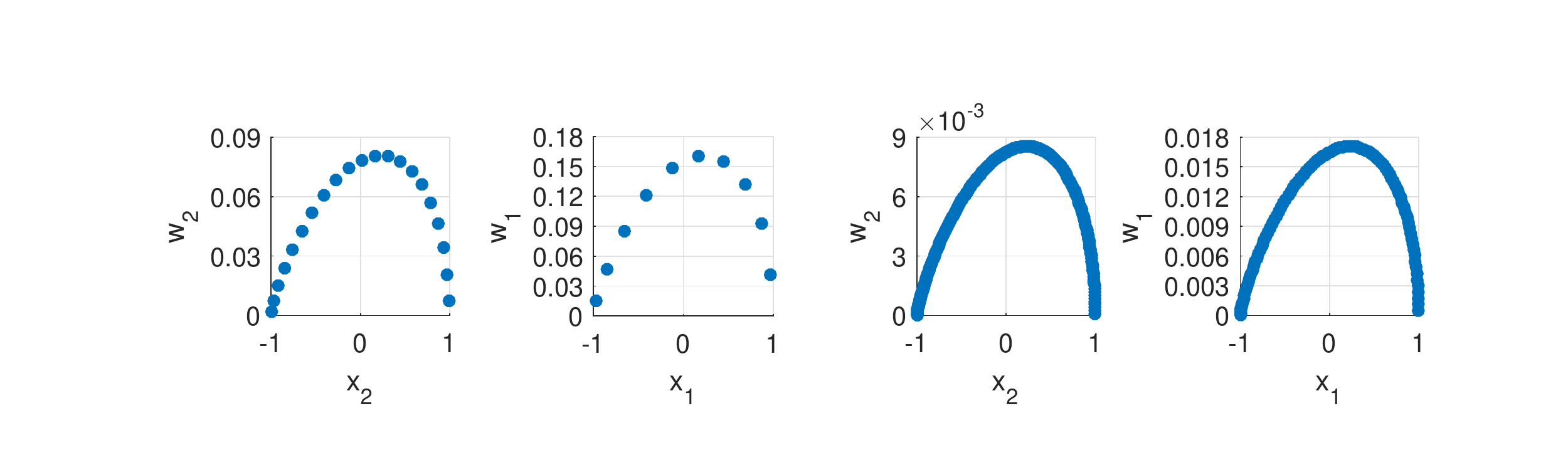}\\
\caption{\small{Optimal main and nested rules for asymmetric Jacobi with $\alpha=0, \beta=0.3$ for $n_1=10, n_2=21, \alpha_1=19, \alpha_2=31$ (two left figures) and $n_1=100, n_2=201, \alpha_1=199, \alpha_2=301$ (two right figures).}}\label{jac}
\end{figure}

Now that we have quadrature nodes for high polynomial order we can test the well-documented \textit{Circle Theorem}. The theorem states that the Gaussian weights, suitably normalized and plotted against the Gaussian nodes, lie asymptotically for large orders on the upper half of the unit circle centered at the origin in the case of Jacobi weight functions~\cite{gautschi_circle_2006}. In other words,
\begin{equation*}
\frac{nw}{\pi \omega(x)} \sim \sqrt{1-x^2} \quad \textrm{as} \quad n \rightarrow~\infty
\end{equation*}
where $x,w$ are the quadrature node and weight and $\omega(x)$ is the weight function evaluated at node $x$. Figure~\ref{circle_leg} shows the above relationship for both main $n_2=201$ and nested $n_1=100$ rules associated with Legendre polynomial cf. Figure~\ref{leg-cheb}.

\begin{figure}[h]
\centering
\includegraphics[width=4in]{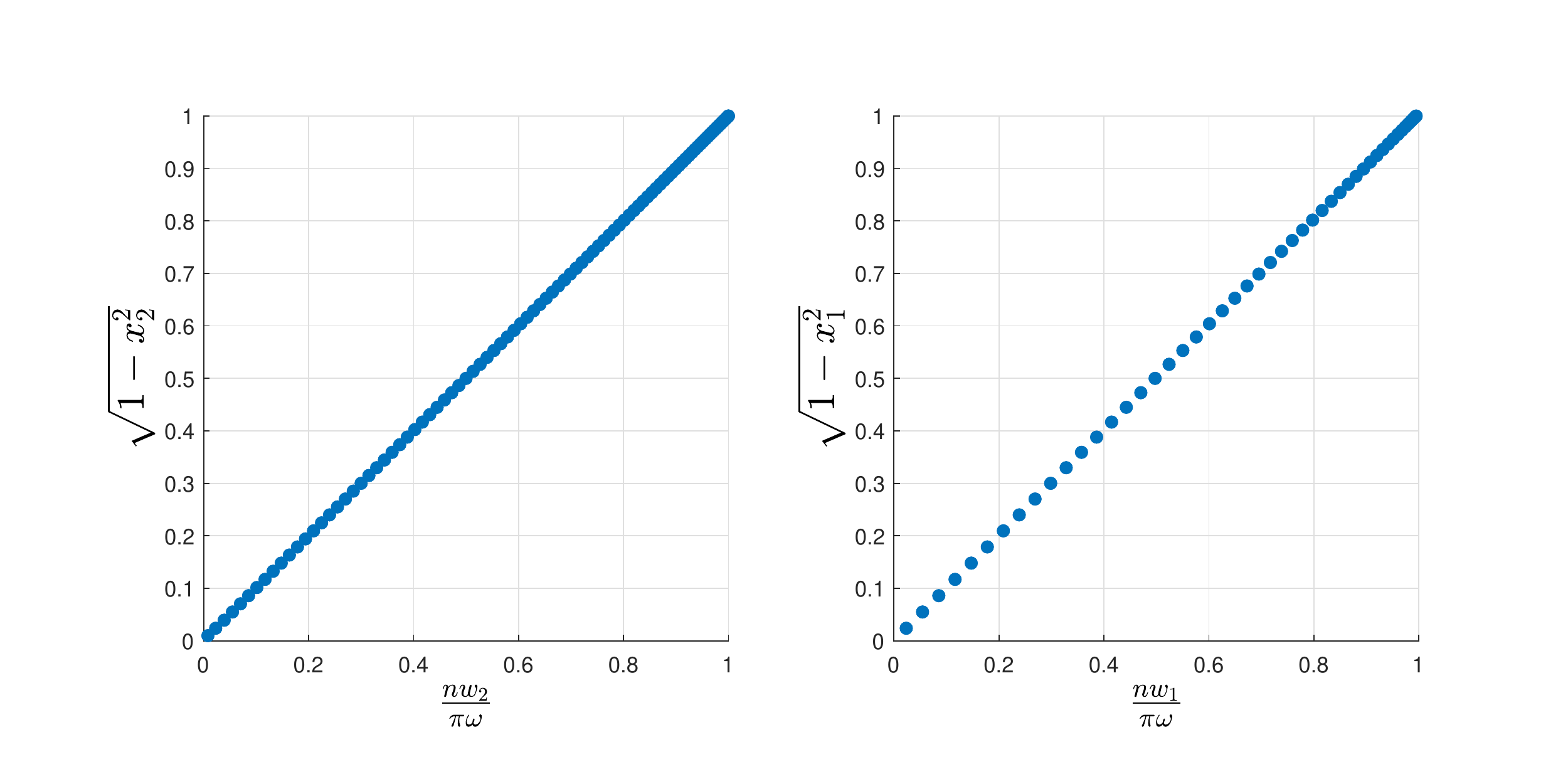}\\
\caption{\small{Circle Theorem for main and nested rule associated with Legendre polynomial.}}\label{circle_leg}
\end{figure}

%

The Gauss-Kronrod-Patterson rule associated with the Legendre polynomial is shown in Figure~\ref{leg-patt_91}. As mentioned earlier we find the nodes for highest order in this case without using the initial guess that we used to generate nodes in Figure~\ref{leg-patt}.

\begin{figure}[h]
\centering
\includegraphics[width=6in]{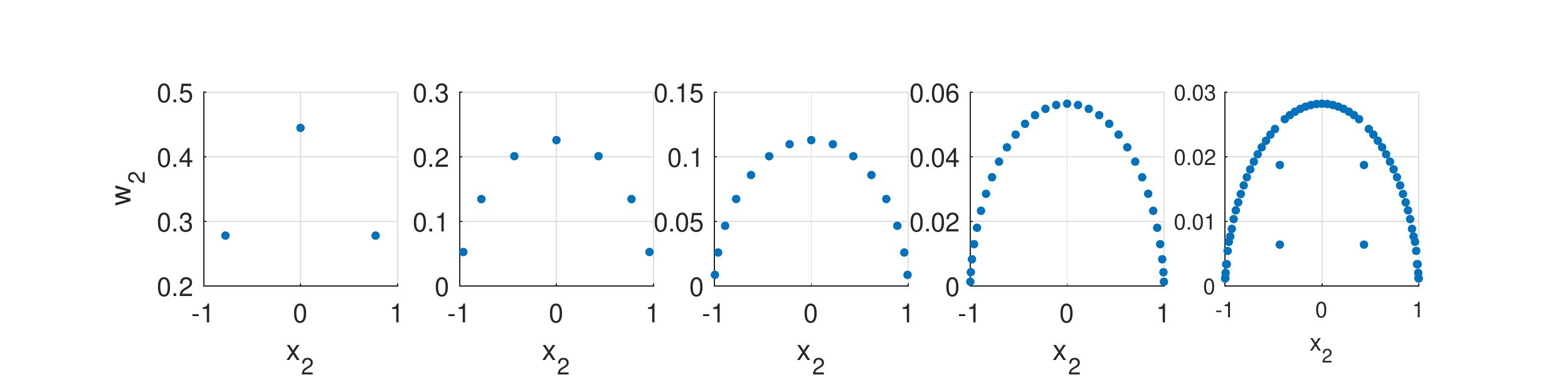}\\
\caption{\small{Optimized nested sequence of quadrature rules $n_2=3,7,15,31,63$ for Legendre polynomial with orders $\alpha_2=5,11,23,47,91$.}}\label{leg-patt_91}
\end{figure}

We also used our method to generate a sequence of nested points for quadrature under the Chebyshev weight. We generated these points for tolerance $\epsilon=10^{-12}$ which attains polynomial accuracy similar to the Legendre case, cf. Figure~\ref{cheb-patt}. For the last rule we started with $63$ points, and we found after optimization that six of these points have negligible weights ($w \sim 10^{-15}$ which is comparable to our constraint tolerance). These points were automatically flagged for removal, so the final set has $57$ points. The residual norm with these $57$ points is $||\bm R_2||_2=1.51e-14$.

\begin{figure}[h]
\centering
\includegraphics[width=6in]{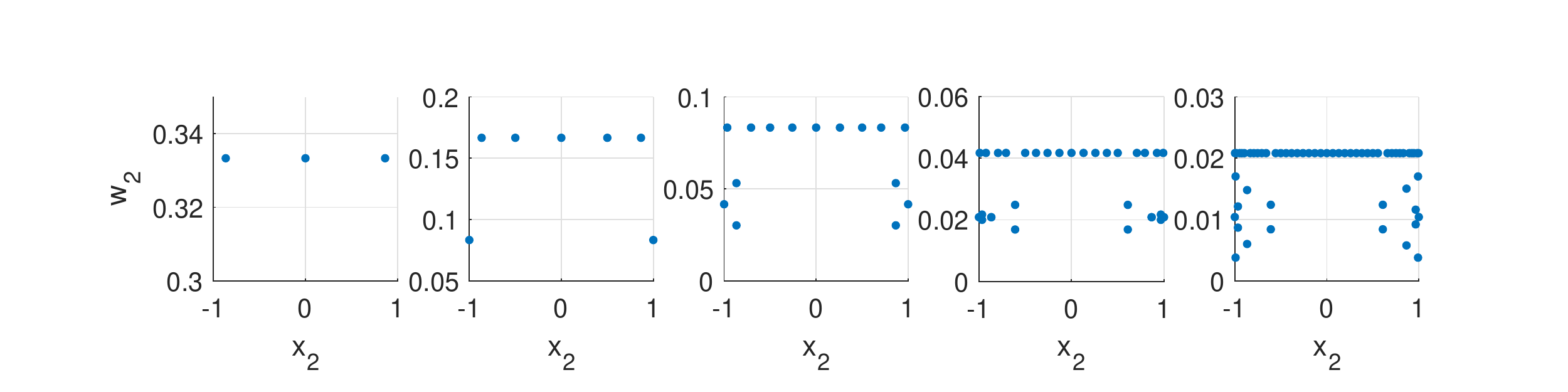}\\
\caption{\small{Optimized nested sequence of quadrature rules $n_2=3,7,15,31,57$ for Chebyshev polynomial with orders $\alpha_2=5,11,23,47,95$.}}\label{cheb-patt}
\end{figure}

\subsubsection{Generation of New Rules: Unbounded Domains}\label{new_rules_UBD}

We now use our numerical method to find nested rules for polynomial families whose orthogonality measure has support on infinite, domains such as the Hermite and Laguerre families.

The weight functions for these two cases are $\omega(x) \varpropto x^{\rho_G} e^{-x^2}$ and $\omega(x)\varpropto{x^{\rho_L}}e^{-x}$. It is easy to show that these two weight functions are transformable to each other with $x_L = x^2_G$ where $x_L$ and $x_G$ denote the domain for weight functions associated with Laguerre and Hermite polynomials. Using this transformation we can show $\rho_L = (\rho_G-1)/2$. Having this transformation and having the ability to generate quadrature for any $\rho_L$ or $\rho_G$ one only needs to generate quadrature points for one of these families. For example generating an $n$-point rule where $n$ is even for Hermite families is equivalent to generating an $n/2$-point rule for Laguerre families. However, it should be noted that the maximum integrable order for Laguerre is half of the Hermite case due to the $x_L = x^2_G$ transformation.


We find in our numerical experiments that the relationship between the number of points and the polynomial accuracy of the rule does not attain the accuracy of a Kronrod rule, i.e., given $n_1$ we do not achieve the parameter $n_2$, $\alpha_1$, and $\alpha_2$ specified in \eqref{eq:kronrod-relationship}. Table~\ref{tabNE_Hermite} lists different values of the number of nodes and polynomial orders for successful cases i.e. cases that achieve tolerance less than $10^{-14}$. The optimization for highest case in this table $n_1=15$ required $1755$ iterations and $5.88~sec$ to achieve the desired tolerance. The results in the table demonstrate that in many cases the accuracy $\alpha_2$ of the main rule is smaller than $3n_1+1$ or $3n_2+2$.
\begin{table}[!h]
\caption{Nested rule for Hermite polynomials with $\rho_G=0$}
\centering
\begin{tabular}{ c c c| c c| c c| c  c c}
\hline\hline
& $i=1$ & $i=2$  & $i=1$ & $i=2$&  $i=1$ & $i=2$ & $i=1$ & $i=2$\\
\hline
$n_i$  & 1 & 3  & 2 & 5 & 3 & 7 & 4 & 9 \\
$\alpha_i$  & 1 & 5  & 3 & 7 & 5 & 9 & 7 & 11\\
\hline\hline
$n_i$   & 5 & 11   & 6 & 13 & 7 & 15  & 8 & 17 \\
$\alpha_i$   & 9 & 15 & 11 & 17 & 13 & 19 & 15 & 21\\
\hline\hline
$n_i$   & 9 & 19 & 10 & 21 & 11 & 23  & 12 & 25  \\
$\alpha_i$   & 17 & 23 & 19 & 25 & 21 & 27 & 23 & 31  \\
\hline\hline
$n_i$  & 13 & 27   & 14 & 29  & 15 & 31\\
$\alpha_i$  & 25 & 33 & 27 & 35 & 29 & 37\\
\hline\hline
\end{tabular}
\label{tabNE_Hermite}
\end{table}

For cases with large number of nodes however we find that we can achieve higher order while maintaining a small tolerance, e.g., $\epsilon=10^{-14}$. This might be attributed to the large number of degrees of freedom in this optimization as well as extremely small weights on the tails. Figure~\ref{herm_high_low} shows the optimal nodes and weights for two cases of a small number of nodes $n_1=15, n_2=31, \alpha_1=29, \alpha_2=37$ and high number of nodes $n_1=100, n_2=201, \alpha_1=199, \alpha_2=301$ for Hermite polynomial. The result for  $n_1=100$ takes $1977$ iterations and $177.32~sec$. With our algorithm we can similarly repeat this experiment for $\rho_G=1$ where we have shown the results in Figure~\ref{herm_high_low_rho1}.


\begin{figure}[h]
\centering
\includegraphics[width=6in]{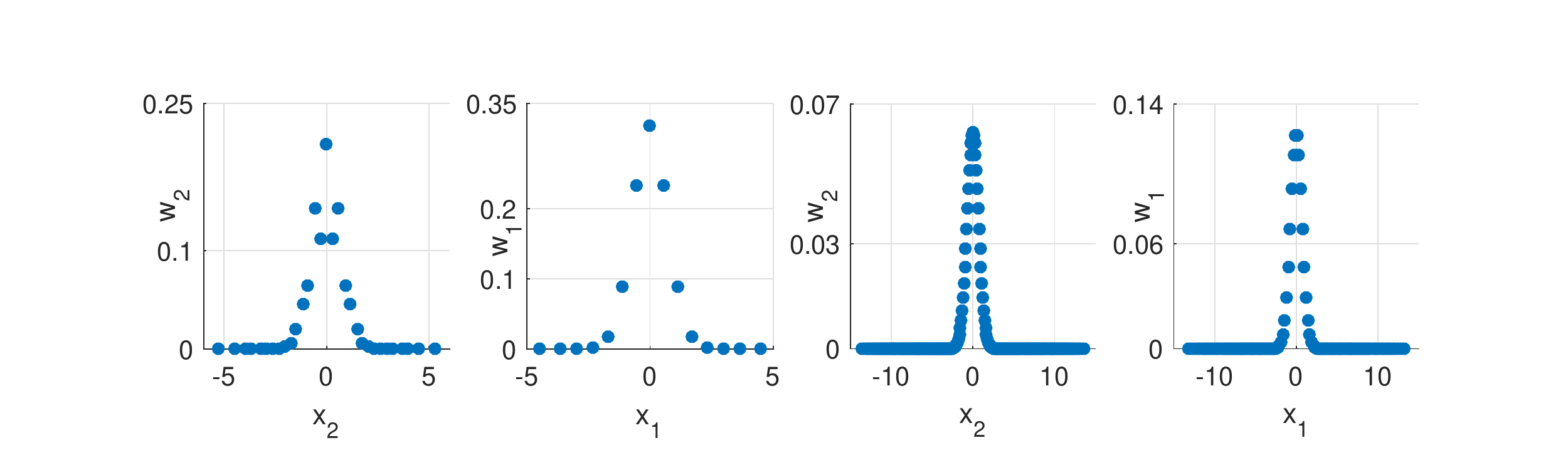}\\
\caption{\small{Optimal main and nested rules for Hermite for $n_1=15, n_2=31, \alpha_1=29, \alpha_2=37$ (two left figures) and $n_1=100, n_2=201, \alpha_1=199, \alpha_2=301$ (two right figures) with $\rho_G=0$.}}\label{herm_high_low}
\end{figure}

\begin{figure}[h]
\centering
\includegraphics[width=6in]{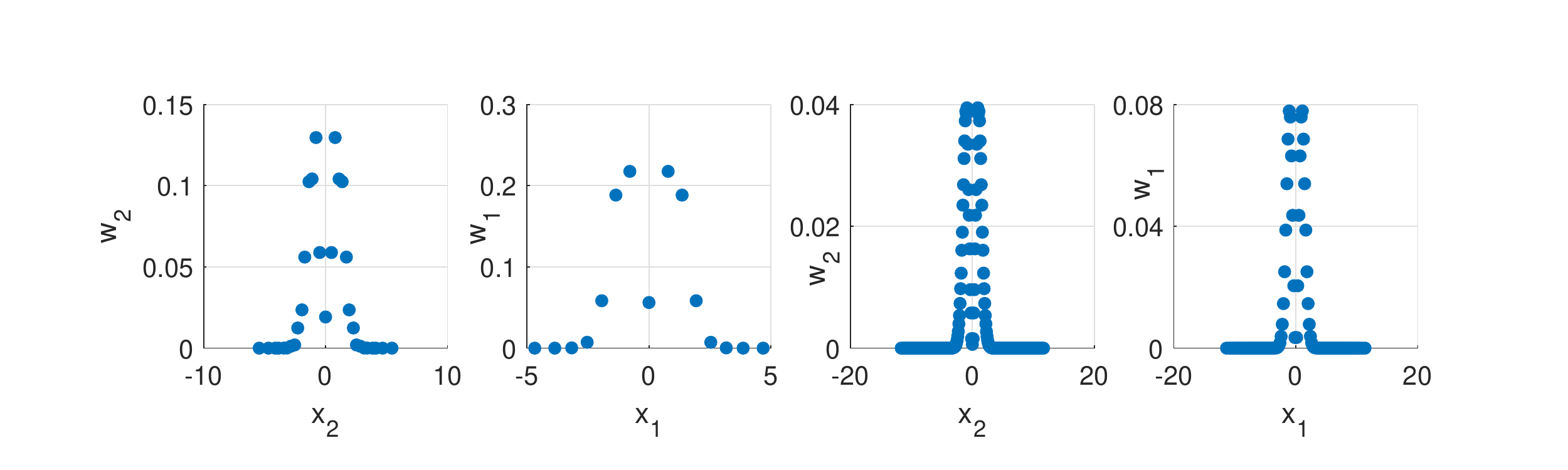}\\
\caption{\small{Optimal main and nested rules for Hermite for $n_1=15, n_2=31, \alpha_1=29, \alpha_2=37$ (two left figures) and $n_1=100, n_2=201, \alpha_1=199, \alpha_2=301$ (two right figures) with $\rho_G=1$.}}\label{herm_high_low_rho1}
\end{figure}


As discussed previously, we use the half of the optimized Hermite rule as the Laguerre rule.  Figure~\ref{lag} shows the optimized Hermite rule $n_1=100, n_2=202, \alpha_1=199, \alpha_2=301$ and the Laguerre rule $n_1=50, n_2=101, \alpha_1=99, \alpha_2=150$ which takes $56.21~sec$ and $3253$ iterations to achieve a residual norm $||\tilde{\bm R}||_2\simeq10^{-14}$. Note that $n_2$ is even in the case of Hermite rule and the accuracy of Laguerre rules are halved. It is also noted that the absicca for these rules are according to $x_L=x^2_G$.


\begin{figure}[h]
  \begin{center}
    \includegraphics[width=\textwidth]{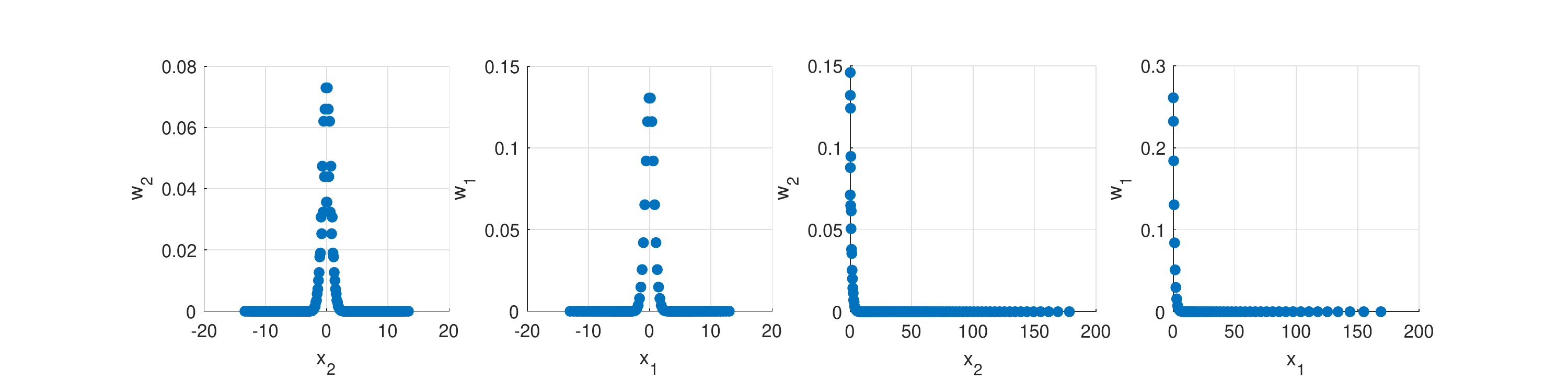}\\
    \includegraphics[width=\textwidth]{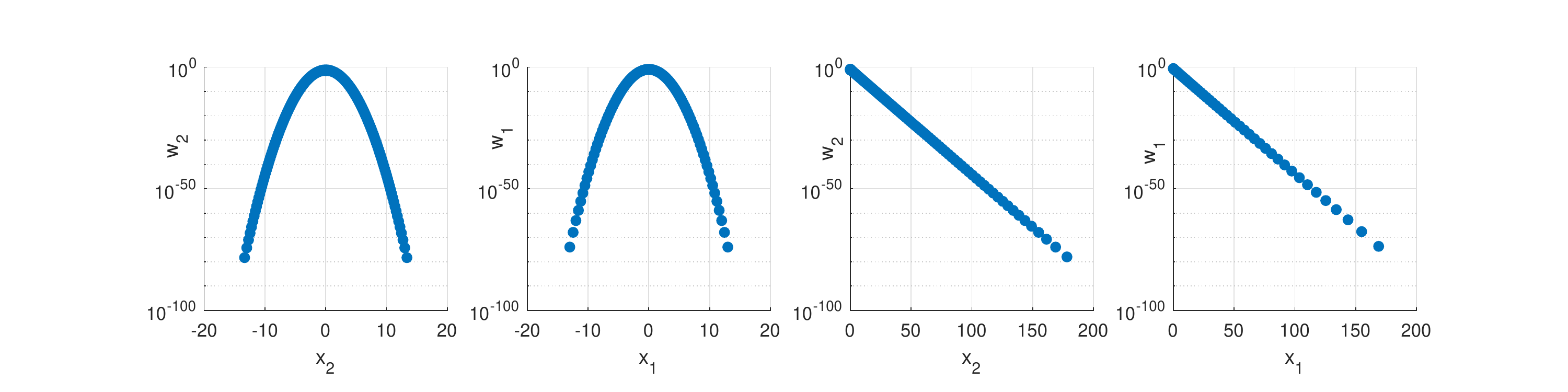}
  \end{center}
\caption{\small{Optimized Hermite rule for $n_1=100, n_2=202, \alpha_1=199, \alpha_2=301$ (two left figures) and  Laguerre nested rules for $n_1=50, n_2=101, \alpha_1=99, \alpha_2=150$ (two right figures). The top and bottom figures show a linear and logarithmic scale, respectively, for the vertical axis.}}\label{lag}
\end{figure}

Finally we generate the Gauss-Kronrod-Patterson rule for Hermite polynomials with $\rho_G=1$. The achieved orders for $n_2=3,7,15,31$ are $\alpha_2=5,9,15,35$ respectively. We use these nodes in Section~\ref{S4_1_5} to generate a Sparse Grid for integration in multiple dimensions.

\begin{figure}[h]
\centering
\includegraphics[width=6in]{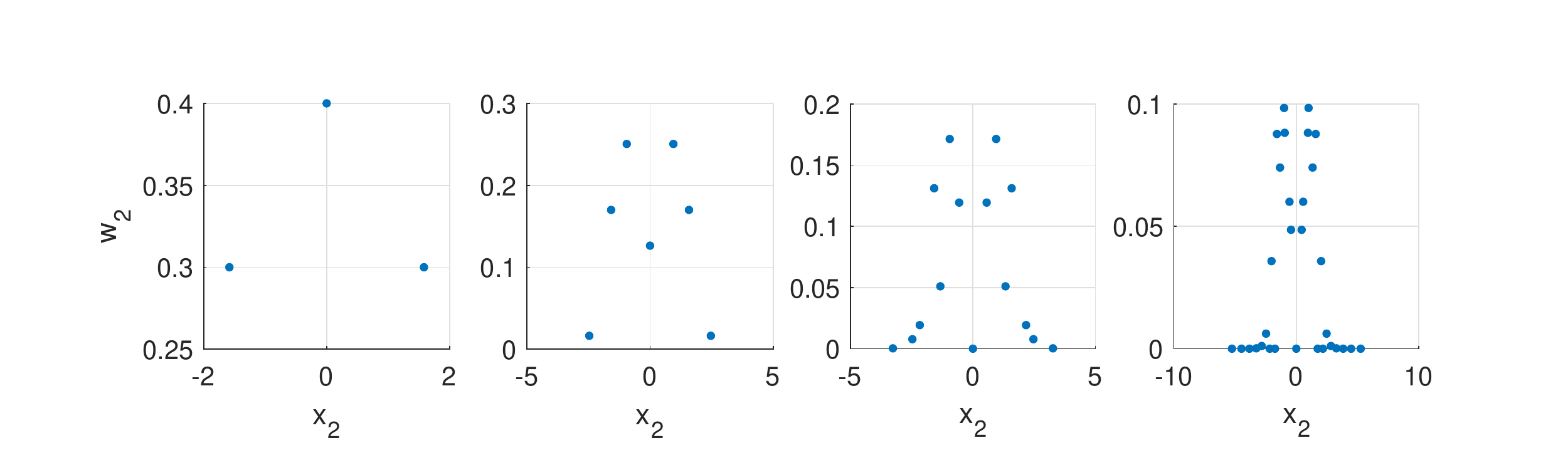}\\
\caption{\small{Optimized nested sequence of quadrature rules $n_2=3,7,15,31$ and $\alpha_2=5,9,15,35$ for Hermite polynomials with $\rho_G=1$.}}\label{leg-patt}
\end{figure}

\subsubsection{Univariate Integration: Linear Elastic Problem}\label{S4_1_4}

We investigate the accuracy of our nested quadrature rule via estimation of statistical moments for the displacement in a linear elastic structure with uncertain material properties. The L-bracket domain shown in Figure~\ref{fig_2_2} is partitioned into $978$ standard triangular elements that are used to discretize a linear elastic PDE that predicts displacement. The modulus of elasticity is parameterized as one lognormal variable $E=10^{-6}+\exp({\xi})$ for all elements. We are interested in the displacement, $u$, in the direction of point load, see Figure \ref{fig_2_2}. This displacement is a function of the elasticity, so that $u = u(\xi)$, where $\xi$ is taken as a standard normal random variable. We estimate the mean and variance of $u(\xi)$ by generating nested quadrature rules in the $\xi$ variable and evaluating $u$ at the abscissae of these rules. Each evaluation of $u$ requires the solution of a PDE, so this is an example of a case where parsimony of quadrature rules sizes is useful. The mean and variance of $u$ are estimated via nested quadrature rules and are compared against ``true'' values, which are computed using a $100$-point Gaussian quadrature rule.

For a given quadrature rule $(\xi_i, w_i)_{i=1}^n$, the mean and standard deviation are estimated as
\begin{equation}\label{eq_m_s}
\begin{array}{l l}
\mu=\displaystyle \sum_{i=1}^n u(\xi_i) w_i, & \sigma=\displaystyle \sqrt{ \displaystyle \sum_{i=1}^n u^2(\xi_i) w_i - \mu^2}

\end{array}
\end{equation}
Subsequently the errors in mean and standard deviation are obtained as
\begin{equation}\label{eq_rel_err}
e_{\mu}=|(\mu-\mu_{true})/\mu_{true}|, \quad  e_{\sigma}=|(\sigma-\sigma_{true})/\sigma_{true}|
\end{equation}
where $\mu_{true}$ and $\sigma_{true}$ are the ``true" mean and standard deviation. We also compute the error between the main and nested rule evaluations as
\begin{equation}\label{I_rel_err}
e_{I}=|(\mu_1-\mu_2)/\mu_{2}|
\end{equation}
where $\mu_1$ and $\mu_2$ are mean values evaluated with nested and main rule respectively.

\begin{figure}[h]
\centering
\includegraphics[width=3.50in]{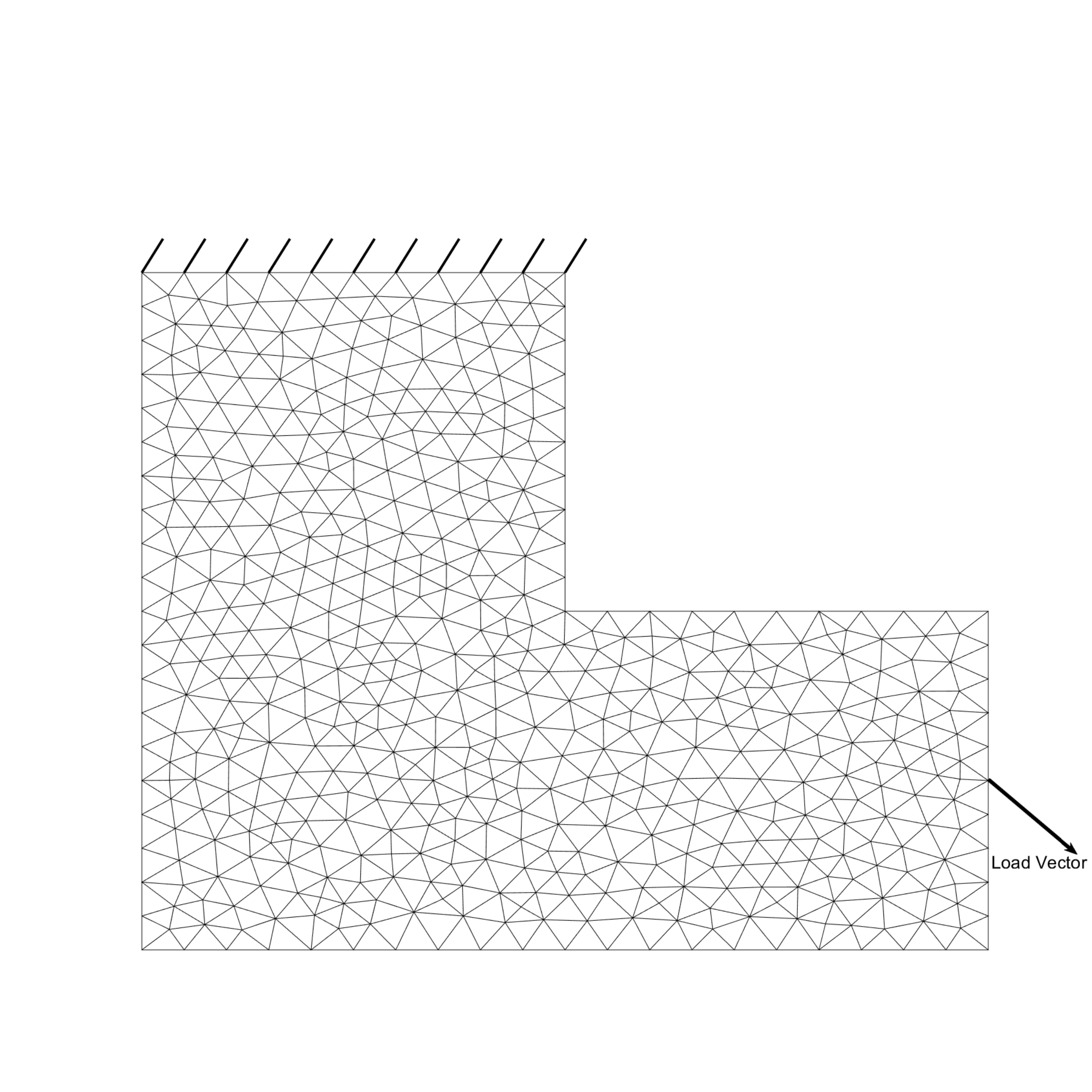}\\ \vspace{-0.5cm}
\caption{\small{Linear elastic L-bracket with random Young's modulus}}\label{fig_2_2}
\end{figure}

The Poisson's ratio is $\nu=0.3$, and the plane stress condition is assumed. We consider two sets for our analysis: i) a nested rule with $n_1=7, n_2=15$ set and compare with an $15$-point Gauss quadrature which is the total number of points and ii) a nested rule with $n_1=8,n_2=17$ and compare with $17$-point Gauss quadrature. We choose the Gauss quadrature sets such that their number is equal to $n_2$.

It is evident from Table~\ref{tabNE_Hermite} that the main ($n_2$-point) quadrature rules have less accuracy compared to $n_2$-point Gaussian quadrature. This is expected since Gaussian rules integrate higher-degree polynomials exactly. However, we see that the nested quadrature rules attain comparable accuracy for $e_\mu$ and $e_I$, which thus supports usage of these rules in cases when re-use of function evaluations is paramount.


\begin{table}[!h]
\caption{Error in mean and standard deviation for the linear elastic structure}
\centering
\begin{tabular}{ c c c c c}
\hline\hline
Quadrature rule & $e_{\mu}$ & $e_{\sigma}$  & number of nodes & $e_{I}$ \\
\hline

Nested nodes  & 4.375e-06&8.757e-03 & 7 & \multirow{ 2}{*}{4.3752e-06}\\
Main nodes  & 5.745e-10 &	8.240e-05  & 7$\oplus$8=15&\\
Gauss quadrature & 4.9141e-12&	1.2708e-08  & 15&\\
\hline
Nested nodes  & 2.918e-07&	2.333e-03  & 8& \multirow{ 2}{*}{2.9185e-07}\\
Main nodes  & 2.550e-11&	1.485e-05  & 8$\oplus$9=17&\\
Gauss quadrature & 3.6890e-12&	2.1421e-10 & 17&\\
\hline\hline
\end{tabular}
\label{tabNE_Hermite}
\end{table}

\subsection{Multivariate examples via sparse grids}

We compute multivariate integration formulas via sparse grids, which manipulate univariate quadrature rules to form a multivariate quadrature rule. The ability to generate nested univariate quadrature rules, which is the main topic of this paper, yields sparse grid constructions that have a relatively small number of function evaluations. This idea is not new, but our procedure affords flexibility: we can generate nested rules for quite general univariate weight functions. We demonstrate the savings using this strategy on some test cases.

\subsubsection{Sparse Grids for Multivariate Quadrature}

Sparse grids are multivariate quadrature rules formed from unions of tensorized univariate rules. Consider a tensorial $\Gamma$ as in Section \ref{sec_notation}, and for simplicity assume that the univariate domains $\Gamma_j = \Gamma_1$ and  weights $\omega_j = \omega_1$ are the same. Let $\mathbb{X}_{i}$ denote a univariate quadrature rule (nodes and weights) of ``level" $i \geq 1$, and define $\mathbb{X}_0 = \emptyset$. The number of points $n_i$ in the quadrature rule $\mathbb{X}_i$ is increasing with $i$, but can be freely chosen. For multi-index $\bm{i} \in \N^d$, a $d$-variate tensorial rule and its corresponding weights are
\begin{equation}
\label{SP1} \displaystyle \mathbb{A}_{d,\bm i}= \mathbb{X}_{i_1} \otimes \ldots \otimes \mathbb{X}_{i_d}, \quad \displaystyle w^{(\bm q)}= \prod_{r=1}^d w_{i_r}^{(q_r)}.
\end{equation}
The difference between sequential univariate levels is expressed as
\begin{align}\label{SP2}
  \Delta_i & = \mathbb{X}_{i} - \mathbb{X}_{i-1}, & i &\geq 1,
\end{align}
This approximation difference is used to construct a $d$-variate, level-$k$-accurate sparse grid operator \cite{Bungartz04,Smol63} for any $k \in \N$ as,
\begin{align}\label{SP3}
  \mathbb{A}_{d,k} = \sum_{r=0}^{k-1} \sum_{\substack{\bm{i} \in \N^d \\\left| \bm{i} \right| = d+r}} \Delta_{i_1} \otimes \ldots \otimes \Delta_{i_d}
   = \sum_{r=k-d}^{k-1} (-1)^{k-1-r}  \binom{d-1}{k-1-r} \displaystyle \sum_{\substack{\bm{i} \in \N^d \\\left| \bm{i} \right| = d+r}} \mathbb{X}_{i_1} \otimes \ldots
\otimes \mathbb{X}_{i_d},
\end{align}
where the latter equality is shown in \cite{Wasilkowski95}.

If the univariate quadrature rule $\mathbb{X}_i$ exactly integrate univariate polynomials of order $2i-1$ or less, then the Smolyak rule $\mathbb{A}_{d,k}$ is exact for $d$-variate polynomials of total order $2k-1$ \cite{Heiss08}. It is reasonable to use Gauss quadrature rules for the $\mathbb{X}_i$ to obtain optimal efficiency, but since the differences $\Delta_i$ appear in the Smolyak construction, then utilizing nested rules satisfying $\mathbb{X}_i \subset \mathbb{X}_{i+1}$ can generate sparse grids with many fewer nodes than non-nested constructions. One can use, for example, nested Clenshaw-Curtis rules \cite{xiu_high-order_2005}, the nested Gauss-Patterson or Gauss-Kronrod rules \cite{Patterson68,liu_adaptive_2011,gerstner_numerical_1998}, or Leja sequences~\cite{narayan_adaptive_2014}.

Sparse grids is a popular rule for integration in many computational applications. The main reason is the easy construction of multidimensional rule from a univariate rule while yielding small number of points. As mentioned sparse grid construction results in fewer nodes by using nested univariate rules. Another alternative to sparse grid for integration in multi-dimensions is the \textit{designed quadrature} which directly satisfies moment-matching conditions for multidimensional polynomial spaces, guarantees all positive weights and has been shown to use far fewer nodes for integration for the same level of accuracy~\cite{Keshavarzzadeh_DQ2017}.

Our multidimensional sparse grid rules are constructed via \eqref{SP3}, but with tensorized quadrature rules formed via $\mathbb{X}_i$ that only approximately integrate polynomials. I.e., the univariate rules $\mathbb{X}_i$ only integrate polynomials up to the accurate certified by $\|\bs{R}\|_2 \leq \epsilon$ from the optimization \eqref{S3_3_3}. This univariate error translates into an error committed for multivariate quadrature rules. For simplicity, we state this result for a tensorial probability density function with identical univariate marginals.
\begin{proposition}
  Assume $\omega(x)$ is a univariate probability density function. Let $\mathbb{X}_i$, $i = 1, \ldots, $ be a sequence of univariate quadrature rules, and for each $i$ assume that the residual vector $\bs{R}$ defined in \eqref{S3_1_0} and \eqref{eq:R-def} satisfies $\|\bs{R}\|_2 < \epsilon$, where the residual vector for $\mathbb{X}_i$ is associated with a univariate polynomial space $\Pi_{\alpha_i}$. Then, given some multi-index $\bs{i} \in \N_0^d$, we have for any $p \in \otimes_{q=1}^d \Pi_{\alpha_{i_q}}$,
  \begin{align*}
    \left| \mathbb{A}_{d,\bs{i}}(p) - I(p) \right| \leq \epsilon \|p\| d (1+\epsilon)^{d-1} \prod_{q=1}^d \sqrt{\alpha_{i_q}+1},
  \end{align*}
  where
  \begin{align*}
    I(p) \coloneqq \int_{\Gamma} p(\bs{x}) \left(\prod_{q=1}^d \omega(x_q)\right) \dx{x_1} \cdots \dx{x_d},
  \end{align*}
  and $\|p\|$ is the $\prod_{q=1}^d \omega(x_q)$-weighted $L^2(\Gamma)$ norm.
\end{proposition}
\begin{proof}
  Given a multi-index $\bs{i} \in \N_0^d$, we define
  \begin{align*}
    \bs{\alpha} = \bs{\alpha}_{\bs{i}} &\coloneqq \left( \alpha_{i_{1}}, \ldots, \alpha_{i_d} \right)^T \in \N_0^d, & 
  \end{align*}
Let $p \in \otimes_{q=1}^d \Pi_{\alpha_{i_q}}$. Then there are coefficients $\widehat{p}_{\bs{j}}$ such that
  \begin{align*}
    p(\cdot) = \sum_{\bs{j} \leq \bs{\alpha}_{\bs{i}}} \widehat{p}_{\bs{j}} p_{\bs{j}}(\cdot),
  \end{align*}
  where
  \begin{align*}
    p_{\bs{j}}(\bs{x}) &= \prod_{k=1}^d p_{j_k}(x_k), & \bs{x} &= (x_1, \ldots, x_d)^T \in \R^d.
  \end{align*}
  Then
  \begin{align}\label{eq:prop-temp}
    \left| \mathbb{A}_{d,\bs{i}} (p) - I(p) \right|^2 \leq \left(\sum_{\bs{j} \leq \bs{\alpha}_{\bs{i}}} \left| \widehat{p}_{\bs{j}} \right| \left| \mathbb{A}_{d,\bs{i}}(p_{\bs{j}}) - I(p_{\bs{j}})\right|\right)^2 &\leq \left( \sum_{\bs{j} \leq \bs{\alpha}_{\bs{i}}} \widehat{p}^2_{\bs{j}} \right) \left( \sum_{\bs{j} \leq \bs{\alpha}_{\bs{i}}} \left| \mathbb{A}_{d,\bs{i}}(p_{\bs{j}}) - I(p_{\bs{j}})\right|^2 \right).
  \end{align}
  The first term, by Parseval's equality, is $\|p\|^2$. To bound the second term, we first show that, given $\bs{r}, \bs{s} \in \R^d$ satisfying
  \begin{align}\label{eq:rs-conditions}
    \sup_{q=1, \ldots, d} \left| s_q - r_q\right| &\leq \epsilon, & \sup_{q=1, \ldots, d} |s_q| \leq 1,
  \end{align}
  then $D_k(\bs{s}, \bs{r}) \coloneqq \left|\prod_{q=1}^k s_q - \prod_{q=1}^k r_q \right|$ satisfies
  \begin{align}\label{eqn:D_k}
    D_k(\bs{s}, \bs{r}) &\leq k \epsilon (1 + \epsilon)^{k-1}, & k = 1, \ldots, d.
  \end{align}
  This result can be established by induction, by first noting that $D_1(\bs{s}, \bs{r}) \leq \epsilon$. For some $k \geq 2$  assume $D_{k-1} \leq (k-1) \epsilon (1+\epsilon)^{k-2}$, then
\begin{equation*}
\begin{array}{l l}
D_{k} &= | \prod_{q=1}^k s_q - \prod_{q=1}^k r_q | \\
      \\
      &= | \prod_{q=1}^k s_q - s_k \prod_{q=1}^{k-1} r_q + s_k \prod_{q=1}^{k-1} r_q - \prod_{q=1}^k r_q | \\
      \\
      &= | s_k \left( \prod_{q=1}^{k-1} s_q - \prod_{q=1}^{k-1} r_q \right) + (s_k - r_k) \prod_{q=1}^{k-1} r_q | \\
      \\
      &\leq |s_k| D_{k-1} + |s_k - r_k| \prod_{q=1}^{k-1} |r_q|.
      \end{array}
      \end{equation*}
Since $|s_q| \leq 1$ and $|s_q - r_q| \leq \epsilon$, this implies that $|r_q| \leq 1+\epsilon$. Using the inductive hypothesis
\begin{equation*}
\begin{array}{l l}
D_{k} &\leq  D_{k-1} + \epsilon \prod_{q=1}^{k-1} (1+\epsilon)\\
       \\
      &\leq (k-1) \epsilon (1+\epsilon)^{k-2} + \epsilon  (1+\epsilon)^{k-1}\\
      \\
      &\leq (k-1) \epsilon (1+\epsilon)^{k-1} + \epsilon (1+\epsilon)^{k-1} =    k \epsilon (1+\epsilon)^{k-1}
 \end{array}
      \end{equation*}
yields \eqref{eqn:D_k}. Note then that
  \begin{align*}
    \left| \mathbb{A}_{d,\bs{i}}(p_{\bs{j}}) - I(p_{\bs{j}})\right| = \left| \prod_{q=1}^d \mathbb{X}_{i_q}(p_{j_q}) - \prod_{q=1}^d I(p_{j_q}) \right|.
  \end{align*}
  Since $\omega$ is a probability density, then $|I(p_{j_q})| \leq 1$ for all $j_q$. Furthermore, if the univariate rules $\mathbb{X}_i$ comprising $\mathbb{A}_{d,\bs{i}}$ satisfy the residual condition $\|\bs{R}\|_2 \leq \epsilon$ as in Algorithm \ref{alg:quad-nested}, then
  \begin{align*}
    \left| \mathbb{X}_{i_q}(p_{j_q}) - I(p_{j_q}) \right| = \left| R_{j_q} \right| \leq \|\bs{R}\|_2 \leq \epsilon.
  \end{align*}
  Thus, defining $s_q = \mathbb{X}_{i_q}(p_{j_q})$ and $r_q = I(p_{j_q})$ satisfies \eqref{eq:rs-conditions}, so that
  \begin{align*}
    \left| \prod_{q=1}^d \mathbb{X}_{i_q}(p_{j_q}) - \prod_{q=1}^d I(p_{j_q}) \right| = D_d(\bs{s}, \bs{r}) \leq d \epsilon (1 + \epsilon)^{d-1}.
  \end{align*}
  Using this in \eqref{eq:prop-temp} (and noting the summation has $\displaystyle \prod_{q=1}^d {(\alpha_{i_q}+1)}$ terms) yields the conclusion.
\end{proof}
The above characterization expresses the error committed by a tensorized quadrature rule when the composite univariate rules commit $\epsilon$ error on a particular subspace. Our error bound does not directly translate into an error committed by a sparse grid construction, but it does suggest that sparse grid multivariate qudarature errors can also scale like $\epsilon$. In addition, we observe in the following numerical experiments that our sparse grids constructed from $\epsilon$-approximate univariate grids perform well in practice.

\subsubsection{Multivariate Integration on Sparse Grids: Nonlinear ODE}\label{S4_1_5}
As mentioned previously, sparse grids are a common tool for integration in multiple dimensions. Application of nested quadrature rules in construction of sparse grids are useful since they reduce the total number of nodes in a sparse grid, and come with inexpensive error estimates. In this example we use our nested quadrature rule in construction of sparse grids to estimate the statistical moments for a parameterized nonlinear ordinary differential equation.

We consider the Lotka-Volterra equations, classical predator-prey equations, which are primarily used to describe the dynamics of biological systems. In particular, the evolution of population for species $x$ and $y$ is modeled as
\begin{align*}
\begin{array}{l}
\vspace{0.1cm}
\displaystyle \frac{\partial x}{\partial t}= a x- b xy, \quad x(0)=x_0\\
\displaystyle \frac{\partial y}{\partial t}=cxy - d y, \quad y(0)=y_0
\end{array}
\end{align*}
where $x$ and $y$ are the population of preys and predators and $a,b,c,d$ are modeled as random variables
\begin{align*}
\begin{array}{l l}
\vspace{0.1cm}
a = \exp{(0.1\xi_1)} + 1, & b= 2\exp{(0.1\xi_2)}+0.5\\
c = \exp{(0.1\xi_3)} + 2, & d= 3\exp{(0.1\xi_4)}+1\\
\end{array}
\end{align*}
where the $\xi_i$ are mutually independent and identically distributed random variables, each having distribution with weight $\omega(x)\varpropto x e^{-x^2}$ identical to the weight we used in Gauss-Kronrod-Patterson section~\ref{new_rules_UBD}. The initial population is $x_0=3,y_0=3$. We use a fourth order Runge-Kutta time integration method to simulate the time trajectory of the population for the range $t \in [0,10]$ with the time-step $dt=0.05$. Some solution realizations for the prey population are shown in Figure~\ref{fig_node_err}.


We now estimate the mean and variance of the prey population at time $2$, $x(2)$, which are computed as in \eqref{eq_m_s} via two quadrature rules: i) a sparse grid constructed from univariate nested quadrature rules and ii) the sparse grid constructed with univariate Gauss quadrature rules. We compute the relative error in mean and standard deviation similarly to Equation~\eqref{eq_rel_err} and use a 2881-point in $d=4$ dimensions\cite{Qsparse} to find the true mean and standard deviation.

We follow the sparse grid construction in~\cite{Heiss08} and use $|\mathbb{X}_{1}|=1,~|\mathbb{X}_{2}|=3,~|\mathbb{X}_{3}|=3,~|\mathbb{X}_{4}|=7~,|\mathbb{X}_{5}|=7,~|\mathbb{X}_{6}|=7$-point univariate rules (where $|.|$ denotes the size of set) for accuracy levels $i=1,\ldots,6$. The construction in~\cite{Heiss08} yields a rule for integration of order $2i-1$ corresponding to each level $i$. It should be noted that the three univariate rules i.e. $[1,3,7]$-point rules used in this example are nested consecutively i.e. 1-point rule is nested to the 3-point rule and 3-point rule is nested to 7-point rule as we generated them in Section~\ref{new_rules_UBD}.

The sparse grid construction in $d=4$ dimensions yields $n=[1,9,33,81,193,385]$ and $n=[1,9,41,137,385,953]$ for six accuracy levels corresponding to nested quadrature and Gauss quadrature respectively. 


Figure~\ref{fig_node_err} shows the relative errors in mean and standard deviation with respect to both nested quadrature and Gauss quadrature. It is apparent that using the nested quadrature rule requires smaller number of function evaluations in addition to yielding relatively smaller errors.

\begin{figure}[h]
  \begin{center}
  \resizebox{\textwidth}{!}{
\includegraphics[width=2.5in]{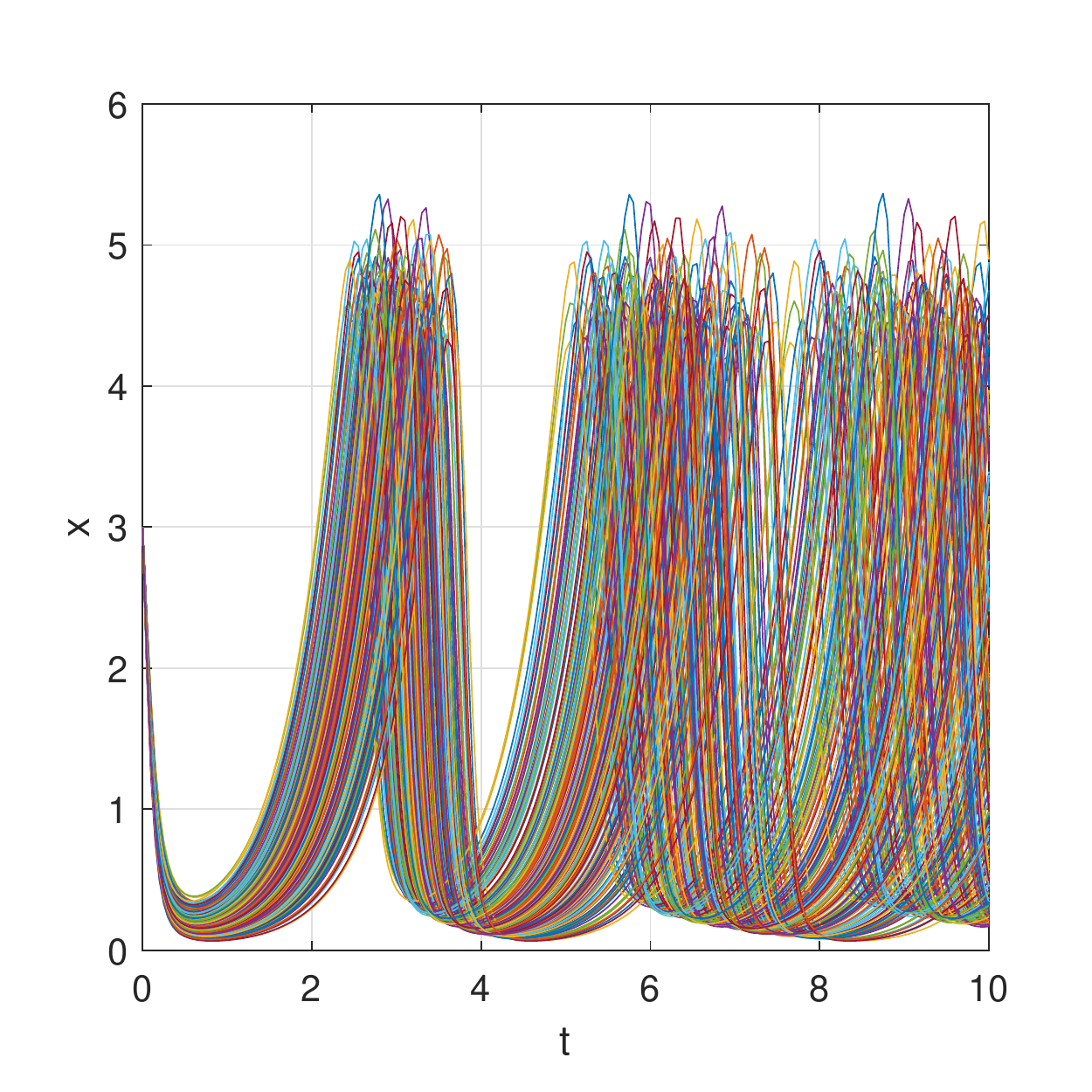} \hspace{-0.64cm}
\includegraphics[width=5.0in]{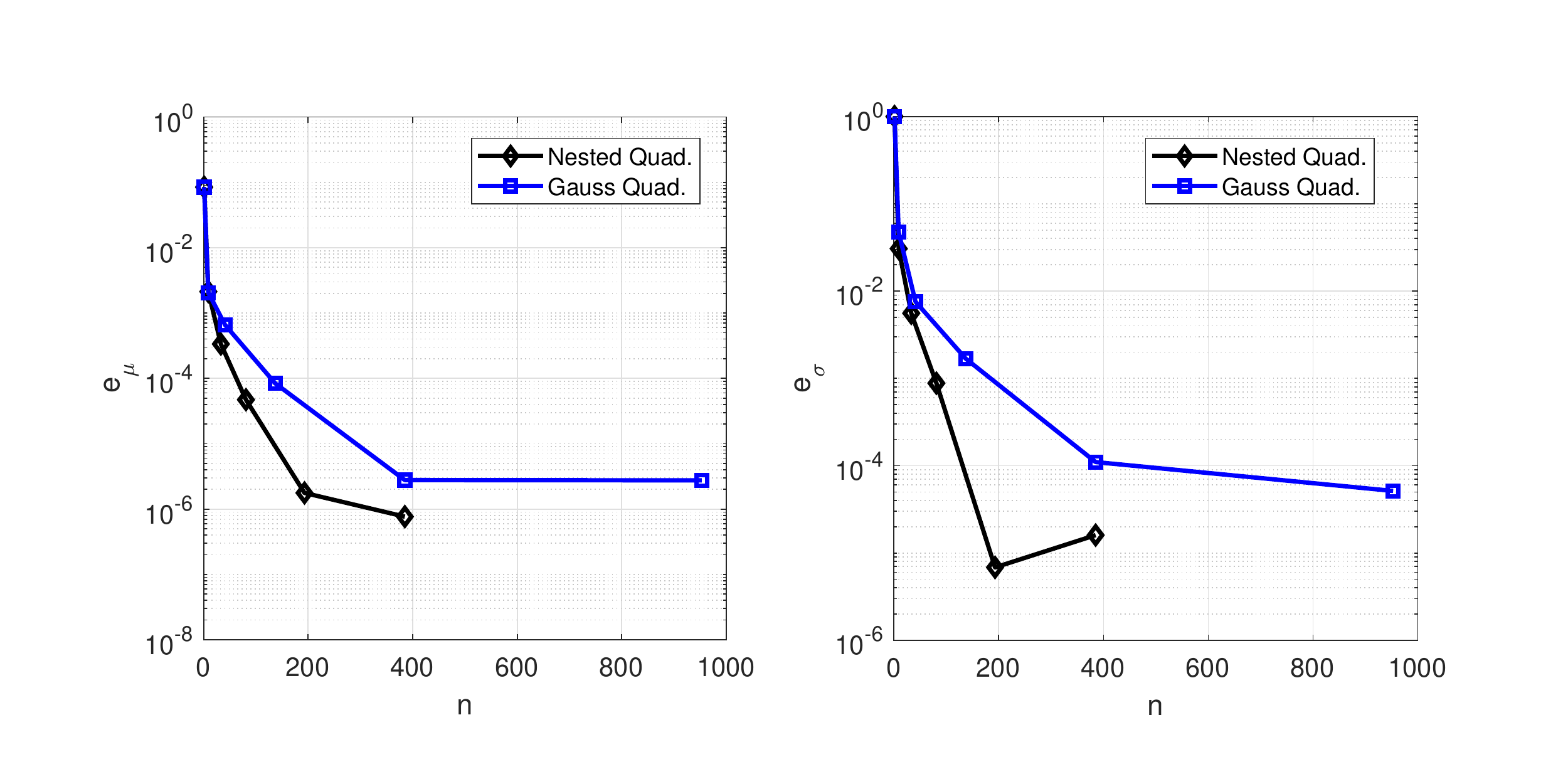}
}
  \end{center}
\caption{\small{Left: Realizations of prey's time history. Center and right: Relative error in mean and standard deviation for the nonlinear ODE}}\label{fig_node_err}
\end{figure}

\subsubsection{Multivariate Integration on Sparse Grids: Elliptic PDE}\label{S4_1_6}
In this example we use sparse grid with nested quadrature to estimate the statistical moment for the steady state heat distribution. Such distribution is modeled via an elliptic PDE with the form
\begin{align*}
\begin{array}{l l}
-\nabla . (a(\bm x,\bm \xi) \nabla u(\bm x,\bm \xi) ) = 1 & \bm x \in \Omega \\
u(\bm x,\bm \xi) = u_0 & \bm x \in \partial \Omega
\end{array}
\end{align*}
where $c$ is the heat conductivity which we consider as a random field in our example. We assume a Karhunen-Loeve expansion in the form of
\begin{align*}
\begin{array}{l l}
  a(\bm x, \bm{\xi}) = \phi_0 + \displaystyle \sum_{i=1}^d \sqrt{\lambda_i} \phi_i(\bm x) \xi_i
\end{array}
\end{align*}
with $\xi_i \sim U[-1,1]$ and $\phi_0$ is a positive constant. The eigenvalues and eigenmodes are obtained from decomposition of a Gaussian covariance kernel
\begin{equation}\label{sm1}
C(\bm x,\bm x')=\exp {\Big (}-{\frac {||\bm x-\bm x'||_2^{2}}{2l_c^{2}}}{\Big )},
\end{equation}
with $l_c=\sqrt{2}/2$.

The spatial domain $\Omega$ and Dirichlet boundary condition are shown in Figure~\ref{fig_mesh_pde}.
\begin{figure}[h]
\centering
\includegraphics[width=5.0in]{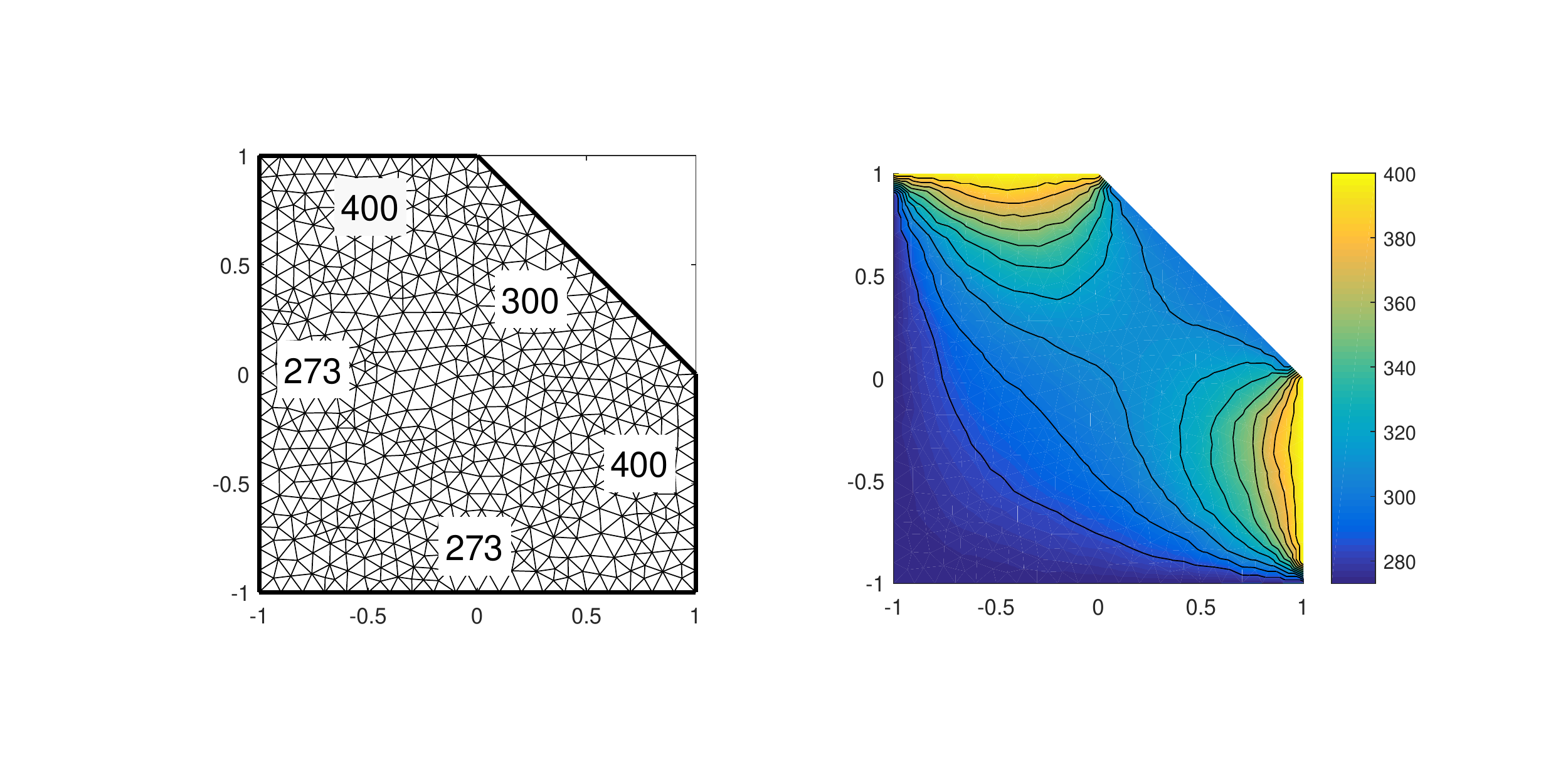}\\
\caption{\small{Finite element mesh with Dirichlet boundary condition (left) and a solution realization for the heat equation (right).}}\label{fig_mesh_pde}
\end{figure}
We truncate the expansion at $d=10$, capturing almost $90\%$ of the energy in the random field, $\sum_{i=1}^{10} \sqrt{\lambda_i}/\sum_{i=1}^{500} \sqrt{\lambda_i} = 0.8825$. The value $\phi_0$ is fixed at $\phi_0=3$.


Similarly to previous example we use three univariate rules and consider [1,3,3,7]-point rules for accuracy levels $i=1,\ldots,4$. The sparse grid construction for $d=10$ results in $n=[1,21,201,1201]$ and $n=[1,21,221,1581]$ nodes for four accuracy levels corresponding to nested quadrature and Gauss quadrature respectively.


Finally, we use a $5281$-point rule for estimating the true mean and standard deviation and focus on a particular node with coordinate $[0.0037,-0.0024]$ in the spatial domain to study the convergence. Figure~\ref{fig_pde_err} shows the relative errors in mean and standard deviation. It is again evident that relatively better accuracy is gained with smaller number of nodes when using a nested quadrature rule.

\begin{figure}[h]
\centering
\includegraphics[width=5.0in]{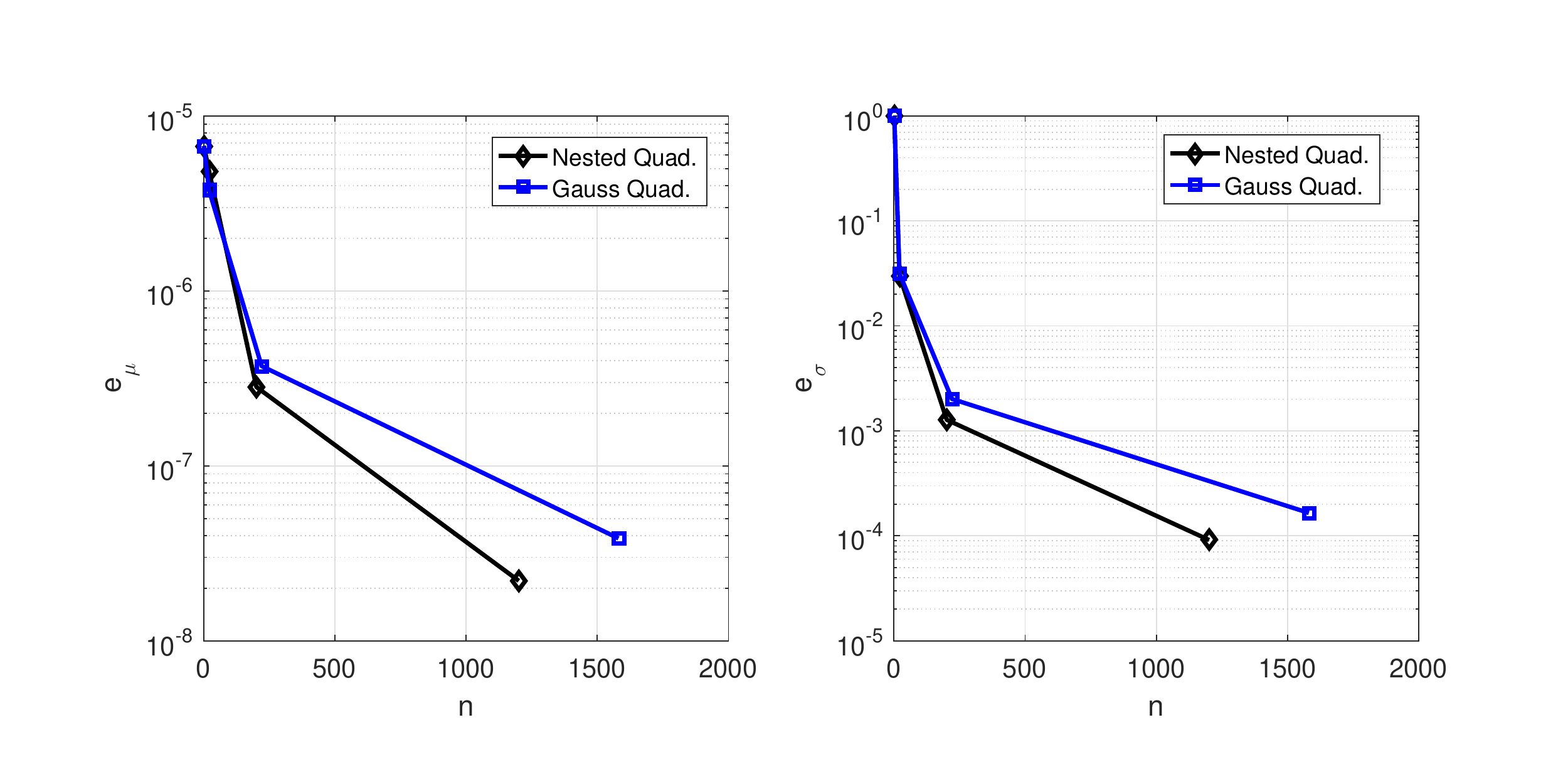}\\ \vspace{-0.5cm}
\caption{\small{Relative error in mean and standard deviation for the elliptic PDE}}\label{fig_pde_err}
\end{figure}

\section{Concluding Remarks}

A numerical method for systematic generation of nested quadrature rules is presented. Our method uses a flexible bi-level optimization that solves the moment-matching conditions for the main and nested rule.  The constraints, namely the node bounds and weight positivity are enforced throughout the optimization via a penalty method.  We generalize the Gauss-Kronrod rule for various weight functions including those with finite/infinite and symmetric/asymmetric supports. The extension of algorithm to generate Gauss-Kronod-Patterson rules i.e. nested sequence of quadrature is also discussed. In particular results for the nested sequence of Chebyshev quadrature are tabulated which have not been reported elsewhere. We used our nested univariate rules to construct sparse grids for integration in multiple dimensions. We showed the improved efficiency and accuracy of the resulting multidimensional quadrature on parameterized initial and boundary value problems when compared with Gauss quadrature-based sparse grids.

\textbf{Acknowledgements}\\
This research was sponsored by ARL under Cooperative Agreement Number W911NF-12-2-0023. The views and conclusions contained in this document are those of the authors and should not be interpreted as representing the official policies, either expressed or implied, of ARL or the U.S. Government. The U.S. Government is authorized to reproduce and distribute reprints for Government purposes notwithstanding any copyright notation herein. The first and third authors are partially supported by AFOSR FA9550-15-1-0467. The third author is partially supported by DARPA EQUiPS N660011524053 and NSF DMS 1720416.

\section*{References}


\begin{thebibliography}{10}
\expandafter\ifx\csname url\endcsname\relax
  \def\url#1{\texttt{#1}}\fi
\expandafter\ifx\csname urlprefix\endcsname\relax\def\urlprefix{URL }\fi
\expandafter\ifx\csname href\endcsname\relax
  \def\href#1#2{#2} \def\path#1{#1}\fi

\bibitem{Kronrod64}
A.~Kronrod, Nodes and weights for quadrature formulae. sixteen place tables,
  Nauka, Moscow, Translation by Consultants Bureau, New York.

\bibitem{szego_orthogonal_1975}
G.~Szeg\"o, Orthogonal {Polynomials}, 4th Edition, American Mathematical Soc.,
  1975.

\bibitem{Stoer2002}
J.~Stoer, R.~Bulirsch, Introduction to numerical analysis, Springer-Verlag New
  York 12.

\bibitem{Davis07}
P.~Davis, P.~Rabinowitz, Methods of numerical integration, Courier Corporation
  2.

\bibitem{Golub69}
G.~Golub, J.~Welsch, Calculation of gauss quadrature rules, Mathematics of
  Computation 23 (1969) 221 --– 230.

\bibitem{Gautschi68}
W.~Gautschi, Construction of {Gauss}-{Christoffel} quadrature formulas,
  Mathematics of Computation 22 (1968) 251--270.

\bibitem{Keshavarzzadeh_DQ2017}
V.~Keshavarzzadeh, R.~M. Kirby, A.~Narayan, Numerical integration in multiple
  dimensions with designed quadrature, SIAM Journal on Scientific Computing
  40~(4) (2018) A2033--A2061.
\newblock \href {http://dx.doi.org/10.1137/17M1137875}
  {\path{doi:10.1137/17M1137875}}.

\bibitem{Bertsekas08}
D.~Bertsekas, Nonlinear programming, Athena Scientific, Second Edition.

\bibitem{Boyd04}
S.~Boyd, L.~Vandenberghe, Convex optimization, Cambridge University Press.

\bibitem{Vandenberg08}
E.~Van~den berg, M.~Friedlander, Probing the {P}areto frontier for basis
  pursuit solutions, SIAM Journal on Scientific Computing 31~(3) (2008)
  890--912.

\bibitem{Vandenberg11}
E.~Van~den berg, M.~Friedlander, Sparse optimization with least-squares
  constraints, SIAM Journal on Optimization 21~(4) (2011) 1201--1229.

\bibitem{golub_matrix_1996}
G.~H. Golub, C.~F.~V. Loan, Matrix {Computations} {Johns} {Hopkins} {Studies}
  in {Mathematical} {Sciences}, 3rd Edition, The Johns Hopkins University
  Press, 1996.

\bibitem{Hansen98}
P.~Hansen, Rank-deficient and discrete ill-posed problems, SIAM, Philadelphia.

\bibitem{Hansen93}
P.~Hansen, D.~O’Leary, The use of the {L-curve} in the regularization of
  discrete ill-posed problems, SIAM Journal on Scientific Computing 14~(6)
  (1993) 1487–--1503.

\bibitem{Patterson68}
T.~Patterson, The optimum addition of points to quadrature formulae,
  Mathematics of Computation 22 (1968) 847--856.

\bibitem{Genz96}
A.~Genz, B.~Keister, Fully symmetric interpolatory rules for multiple integrals
  over infinite regions with gaussian weight, Journal of Computational and
  Applied Mathematics 71~(2) (1996) 299 -- 309.

\bibitem{Qsparse}
F.~Heiss, V.~Winschel, Quadrature on sparse grids,
  \url{http://www.sparse-grids.de/}.

\bibitem{gautschi_circle_2006}
W.~Gautschi, \href{https://eudml.org/doc/127679?lang=it&limit=15}{The circle
  theorem and related theorems for {Gauss}-type quadrature rules.}, ETNA.
  Electronic Transactions on Numerical Analysis [electronic only] 25 (2006)
  129--137.
\newline\urlprefix\url{https://eudml.org/doc/127679?lang=it&limit=15}

\bibitem{Bungartz04}
H.~Bungartz, M.~Griebel, Sparse grids, Acta Numerica 13 (2004) 147 -- 269.

\bibitem{Smol63}
S.~Smolyak, Quadrature and interpolation formulas for tensor products of
  certain classes of functions, Soviet Mathematics Doklady 4 (1963) 240--243.

\bibitem{Wasilkowski95}
G.~Wasilkowski, H.~Wozniakowski, Explicit cost bounds of algorithms for
  multivariate tensor product problems, Journal of Complexity 11~(1) (1995) 1
  -- 56.

\bibitem{Heiss08}
F.~Heiss, V.~Winschel, Likelihood approximation by numerical integration on
  sparse grids, Journal of Econometrics 144~(1) (2008) 62 -- 80.

\bibitem{xiu_high-order_2005}
D.~Xiu, J.~S. Hesthaven, High-{Order} {Collocation} {Methods} for
  {Differential} {Equations} with {Random} {Inputs}, SIAM Journal on Scientific
  Computing 27~(3) (2005) 1118--1139.
\newblock \href {http://dx.doi.org/10.1137/040615201}
  {\path{doi:10.1137/040615201}}.

\bibitem{liu_adaptive_2011}
M.~Liu, Z.~Gao, J.~S. Hesthaven, Adaptive sparse grid algorithms with
  applications to electromagnetic scattering under uncertainty, Applied
  Numerical Mathematics 61~(1) (2011) 24--37.
\newblock \href {http://dx.doi.org/10.1016/j.apnum.2010.08.002}
  {\path{doi:10.1016/j.apnum.2010.08.002}}.

\bibitem{gerstner_numerical_1998}
T.~Gerstner, M.~Griebel, Numerical integration using sparse grids, Numerical
  Algorithms 18~(3) (1998) 209--232.
\newblock \href {http://dx.doi.org/10.1023/A:1019129717644}
  {\path{doi:10.1023/A:1019129717644}}.

\bibitem{narayan_adaptive_2014}
A.~Narayan, J.~Jakeman, Adaptive {Leja} {Sparse} {Grid} {Constructions} for
  {Stochastic} {Collocation} and {High}-{Dimensional} {Approximation}, SIAM
  Journal on Scientific Computing 36~(6) (2014) A2952--A2983, arXiv:1404.5663
  [math.NA].
\newblock \href {http://dx.doi.org/10.1137/140966368}
  {\path{doi:10.1137/140966368}}.

\end{thebibliography}

\end{document}